\documentclass[12pt]{article} 
\usepackage{amsfonts,amsmath,latexsym,amssymb,mathrsfs,amsthm,comment,setspace}

\let\OLDthebibliography\thebibliography
\renewcommand\thebibliography[1]{
	\OLDthebibliography{#1}
	\setlength{\parskip}{0pt}
	\setlength{\itemsep}{0pt plus 0.0ex}
}

\def\numberlikeadb{\global\def\theequation{\thesection.\arabic{equation}}}
\numberlikeadb
\newtheorem{theorem}{Theorem}[section]
\newtheorem{lemma}[theorem]{Lemma}
\newtheorem{corollary}[theorem]{Corollary}

\newtheorem{remark}[theorem]{Remark}
\newtheorem{example}[theorem]{Example}

\usepackage{color} 
\newcommand{\gr}[1]{{\color{blue} #1}}
\allowdisplaybreaks
\usepackage{lscape}
\usepackage{caption}
\usepackage{multirow}
\begin{document}

	\title{
	Bounds for distributional approximation in the multivariate delta method by Stein's method}
	\author{Robert E. Gaunt\footnote{Department of Mathematics, The University of Manchester, Oxford Road, Manchester M13 9PL, UK, robert.gaunt@manchester.ac.uk; heather.sutcliffe@manchester.ac.uk}\:\, and Heather Sutcliff$\mathrm{e}^{*}$}

	\date{} 
	\maketitle
	
	\vspace{-9mm}
	
	
\begin{abstract} We obtain bounds to quantify the distributional approximation in the delta method for vector statistics (the sample mean of $n$ independent random vectors) for normal and non-normal limits, measured using smooth test functions. For normal limits, we obtain bounds of the optimal order $n^{-1/2}$ rate of convergence, but for a wide class of non-normal limits, which includes quadratic forms amongst others, we achieve bounds with a faster order $n^{-1}$ convergence rate. We apply our general bounds to derive explicit bounds to quantify distributional approximations of an estimator for Bernoulli variance, several statistics of sample moments, order $n^{-1}$ bounds for the chi-square approximation of a family of rank-based statistics, and we also provide an efficient independent derivation of an order $n^{-1}$ bound for the chi-square approximation of Pearson's statistic. In establishing our general results, we generalise recent results on Stein's method for functions of multivariate normal random vectors to vector-valued functions and sums of independent random vectors whose components may be dependent. These bounds are widely applicable and are of independent interest. 
\end{abstract}

	
	\noindent{{{\bf{AMS 2010 Subject Classification:}}} Primary 60F05; 62E17

\section{Introduction}

\subsection{The delta method and discussion of main results}

Let $\mathbf{X}_1,\ldots,\mathbf{X}_n$ be independent random vectors with mean vector $\mathbf{0}$. Let $\overline{\mathbf{X}}=n^{-1}\sum_{i=1}^n\mathbf{X}_i$ denote the sample mean. Suppose that the covariance matrix $n^{-1}\Sigma\in\mathbb{R}^{d \times d}$ of $\overline{\mathbf{X}}$ is non-negative definite. Assume that the first $t+1$ partial derivatives of the function $\mathbf{f}:\mathbb{R}^d\rightarrow\mathbb{R}^m$ exist and are such that
 $\frac{\partial^{t}\mathbf{f}}{\prod_{j=1}^{t}\partial w_{i_j}}(\mathbf{0})\neq\mathbf{0}$  for some $1\leq i_1,\ldots,i_t\leq d$, with (in the case $t\geq2$) all lower order partial derivatives of $\mathbf{f}$ evaluated at the mean $\mathbf{0}$ being equal to the zero vector $\mathbf{0}$. Define $\mathbf{T}_{t,n}=n^{t/2}(\mathbf{f}(\overline{\mathbf{X}})-\mathbf{f}(\mathbf{0}))$. Then, roughly stated, the delta method asserts that, for fixed $t\geq1$, the statistic $\mathbf{T}_{t,n}$ converges in distribution to the random vector
\begin{equation}\label{limity}
	\mathbf{Y}_t=\frac{1}{t!}\sum_{i_1,\ldots,i_t=1}^{d}\frac{\partial^t\mathbf{f}}{\prod_{j=1}^{t}\partial w_{i_j}}(\mathbf{0})Z_{i_1}Z_{i_2}\cdots Z_{i_t},
\end{equation}
where  $(Z_1,\ldots,Z_d)^\intercal\sim \mathrm{N}_d(\mathbf{0},\Sigma)$, the $d$-dimensional multivariate normal distribution with mean vector $\mathbf{0}$ and covariance matrix $\Sigma$.
In the case $t=1$, the limit distribution $\mathbf{Y}_1$ is a multivariate normal random vector. We assume that $\mathbb{E}[\overline{\mathbf{X}}]=\mathbf{0}$ for a simpler exposition. Results for the general case $\mathbb{E}[\overline{\mathbf{X}}]=\boldsymbol{\mu}$ are readily obtained by applying a simple translation; see the examples in Section \ref{sec3} for an illustration.

The delta method is one of the most basic and widely used results in the asymptotic theory of mathematical statistics. A statement of the delta method for distributional approximation dates as far back as a work of \cite{d35} in 1935; we refer the reader to \cite{p13,v12} for an interesting historical discussion of the origins of the method. Proofs, detailed accounts and applications of the delta method can be found in, amongst many others, the books and lecture notes \cite{bfh75,h14,s10}, and a treatment of the
delta method applied to infinite-dimensional random vectors is given by \cite{k08,r06}.

Given the importance of the delta method in mathematical statistics, it is perhaps surprising that
the natural problem of seeking error bounds to quantify the quality of the distributional approximation has received surprisingly little attention in the literature. Recently, however, \cite{pm16} (see also the arXiv version \cite{pm16b} for further results) produced a substantial work in which optimal order $n^{-1/2}$ uniform and non-uniform Berry-Esseen bounds were given in the multivariate delta method (with $f:\mathbb{R}^d\rightarrow\mathbb{R}$) in the important case of normal approximation ($t=1$). Even when specialised to the univariate delta method their results were new. Berry-Esseen bounds in the high-dimensional delta method are also given by \cite{wkr14}. The bounds of \cite{pm16} were used by \cite{p17} to establish optimal order $n^{-1/2}$ uniform and non-uniform Berry-Esseen bounds for the normal approximation of the maximum likelihood estimator, whilst \cite{ag20,al17} used the delta method and Stein's method \cite{stein} to obtain $O(n^{-1/2})$ bounded Wasserstein distance bounds for the normal approximation of maximum likelihood estimators that can be expressed as a function of a sum of independent random variables. Recently, \cite{gaunt normal} used Stein's method to obtain bounds for the distributional approximation in the univariate delta method for the cases $t=1,2$, measured using smooth test functions. Notably, \cite{gaunt normal} obtained a bound of order $n^{-1}$ for the case $t=2$ under the assumption that $f$ is an even function (that is $f(x)=f(-x)$ for all $x\in\mathbb{R}$) with derivatives up to fourth order having polynomial growth, in which the limit distribution is a chi-square distribution with one degree of freedom.

In this paper, we obtain general bounds for the distributional approximation in the multivariate delta method, measured using smooth test functions. We obtain bounds on the quantity $|\mathbb{E}[h(\mathbf{T}_{t,n})]-\mathbb{E}[h(\mathbf{Y}_t)]|$, where $h:\mathbb{R}^m\rightarrow\mathbb{R}$ is a real-valued, suitably smooth test function. Bounds on the distance between the distributions of $\mathbf{T}_{t,n}$ and $\mathbf{Y}_t$ with respect to integral probability metrics \cite{gs02,z83} readily follow; see Remark \ref{rem2.2} for further details. Our bounds apply for vector-valued functions $\mathbf{f}:\mathbb{R}^d\rightarrow\mathbb{R}^m$ that are sufficiently differentiable with derivatives having polynomial growth; the sample mean $\overline{\mathbf{X}}$ is a sum of independent, but not necessarily identically distributed, $d$-dimensional random vectors; and the bounds hold for all $t\geq1$. We obtain order $n^{-1/2}$ bounds in this general setting (see Theorem \ref{thm2.1}), as well as bounds with a faster $O(n^{-1})$ convergence rate under certain additional assumptions. These faster convergence rates occur if $t\geq2$ is even and either $\mathbf{f}$ is an even function (that is $\mathbf{f}(\mathbf{x})=\mathbf{f}(-\mathbf{x})$ for all $\mathbf{x}\in\mathbb{R}^d$) or $\mathbb{E}[X_{ij}X_{ik}X_{il}]=0$ for all $1\leq i\leq n$, $1\leq j,k,l\leq d$ (where $X_{ij}$ is the $j$-th component of the random vector $\mathbf{X}_i$), as well as stronger differentiability conditions on $\mathbf{f}$ and stronger moment conditions on the $X_{ij}$ (see Theorems \ref{thm2.2} and \ref{thm2.3}).  The identification of these simple sufficient conditions for faster than $O(n^{-1/2})$ convergence rates in the delta method is one of the main contributions of this paper.  

The work of \cite{pm16} is significant in that the bounds are given in the Kolmogorov metric,
which is  important in statistics, as bounds in this metric can be used, for example, to construct conservative confidence intervals. However, in this paper we reap several benefits by working with smooth test functions. Firstly, the premise of smooth test functions is crucial for our purpose of obtaining faster than order $n^{-1/2}$ convergence rates; see Remark \ref{rem2.1}. Secondly, we can appeal to the powerful theory of Stein's method for functions of multivariate normal random vectors to treat a wide class of limit distributions beyond the normal distribution. Thirdly, as noted by \cite[p.\ 151]{bh85}, bounds for smooth functions may be more natural in theoretical settings.
Our $O(n^{-1/2})$ bounds for the univariate delta method for general $t\geq1$ immediately induce $O(n^{-1/2})$ Wasserstein distance bounds, a natural and widely used probability metric with many applications in statistics (see \cite{pz19}). Moreover, from our smooth test function bounds we can use general results of \cite{gl22} to extract (typically sub-optimal order) Kolmogorov distance bounds for the multivariate delta method; see Remark \ref{rem2.2} for further details. To the best of our knowledge, these are the first Kolmogorov distance bounds for the multivariate delta method that hold for general $t\geq1$.  

In addition to complementing the work of \cite{pm16}, our results  significantly generalise those of \cite{gaunt normal} on the delta method. When specialised to the case $t=2$ and $m=1$, the limit distribution is a chi-square distribution. Our bounds therefore also complement recent works of \cite{gaunt pd, gaunt chi square, gr21} in which Stein's method is used to obtain optimal order $n^{-1}$ bounds for the chi-square approximation of the Pearson, power divergence and Friedman statistics, as well as the works of \cite{asylbekov,gu03,pu21,ulyanov} in which Edgeworth expansions are used to obtain order $n^{-1}$ Berry-Esseen bounds, without explicit constants, for the $\chi_{(d)}^2$ approximation of the Pearson, power divergence and Friedman statistics for sufficiently large $d$. It is noteworthy that a simple application of our general bounds yields independent derivations of explicit $O(n^{-1})$ bounds for the chi-square approximation of the Friedman (as a special case of a result for more general rank-based statistics introduced by \cite{sen}) and Pearson statistics (see Examples \ref{ex3.5} and \ref{ex3.6}).
We also note that Stein's method has been used to obtain error bounds
for the multivariate normal approximation of vectors of quadratic forms \cite{c08,de15}.

\subsection{Discussion of methods and outline of the paper}

We will prove our general bounds for the delta method by generalising recent results of \cite{gaunt normal,GauntSut} on Stein's method for functions of multivariate normal random vectors. Here one is interested in bounding the distance between the distributions of $g(\mathbf{W}_n)$ and $g(\mathbf{Z})$, where $g:\mathbb{R}^d\rightarrow\mathbb{R}$ is a continuous function and $\mathbf{W}_n$ converges in distribution, as $n\rightarrow\infty$, to a multivariate normal random vector $\mathbf{Z}$. From now on, to simplfiy notation, we suppress the dependence of $n$ in $\mathbf{W}_n$. In \cite{gaunt normal,GauntSut}, general bounds on the quantity $|\mathbb{E}[h(g(\mathbf{W}))]-\mathbb{E}[h(g(\mathbf{Z}))]|$ were obtained in the case that $g$ has partial derivatives with polynomial growth and $\mathbf{W}=(W_1,\ldots,W_d)^\intercal$ with $W_j=n^{-1/2}\sum_{i=1}^nX_{ij}$ for independent $X_{ij}$, $i=1,\ldots,n$, $j=1,\ldots,d$. The general bounds of \cite{gaunt normal,GauntSut} yield $O(n^{-1})$ convergence rates if $g$ is an even function (that is $g(\mathbf{w})=g(-\mathbf{w})$ for all $\mathbf{w}\in\mathbb{R}^d$) or if $\mathbb{E}[X_{ij}^3]=0$ for all $i,j$, in addition to the usual first and second moment assumptions. Without these additional assumptions, one obtains bounds of order $n^{-1/2}$.

In this paper, we generalise the results of \cite{gaunt normal,GauntSut} to vector-valued functions $\mathbf{g}:\mathbb{R}^d\rightarrow\mathbb{R}^m$ and also to the case that, for fixed $i$, the $X_{ij}$ may be dependent. The latter generalisation allows us to derive bounds for the delta method for general sample means comprised of sums of independent random vectors, whose components may be dependent; the significance of this generalisation is demonstrated with our application to the rank-based statistics and Pearson's statistic. 
These bounds are themselves widely applicable and of independent interest. Indeed, Stein's method for functions of multivariate normal random vectors has already been used to derive bounds for the chi-square approximation of the likelihood ratio \cite{ar20},  power divergence \cite{gaunt pd} and Friedman statistics \cite{gr21}, as well as multivariate normal approximation of the posterior in exponential family models \cite{fgrs22} and the maximum likelihood estimator \cite{ag20}, and the generalisations established in this paper further broaden the scope of the theory. 
With these bounds at hand, we can bound the quantity of interest $|\mathbb{E}[h(\mathbf{T}_{t,n})]-\mathbb{E}[h(\mathbf{Y}_t)]|$ by applying the triangle inequality to reduce the problem to bounding an expression of the form $|\mathbb{E}[h(\mathbf{g}(\mathbf{W}))]-\mathbb{E}[h(\mathbf{g}(\mathbf{Z}))]|$ and another remainder term that only involves the limit random variable $\mathbf{Y}_t$ which can be bounded by elementary calculations.

The rest of this paper is organised as follows. In Section \ref{sec2.1}, we present our general bounds for the distributional approximation in the delta method together with several remarks discussing the bounds. In Section \ref{sec2.2}, we consider a noteworthy special case in which the $(t+1)$-th order partial derivatives of $\mathbf{f}$ are zero. These bounds generalise recent results of \cite{gaunt normal,GauntSut} on Stein's method for functions of multivariate normal random vectors to vector-valued functions and sums of independent random vectors whose components may be dependent.  In Section \ref{sec3}, we provide several illustrative examples of the application of our general bounds. Our examples concern distributional approximations of an estimator for Bernoulli variance, several statistics of sample moments, and order $n^{-1}$ bounds for the chi-square approximation of a family of rank-based statistics and Pearson's statistic.
 In Section \ref{sec4}, we prove the theorems from Section \ref{sec2.2}; along the way we obtain Lemma \ref{bell lem} which, as discussed in Remark \ref{remm1}, is of particular interest, as it allows bounds the bounds of \cite{gaunt normal,GauntSut} on the derivatives of the solution of the multivariate normal Stein equation with test function $h(g(\cdot))$ to be immediately generalised to vector-valued functions $\mathbf{g}:\mathbb{R}^d\rightarrow\mathbb{R}^m$.   
In Section \ref{sec5}, we leverage the results of Section \ref{sec2.2} to prove the general bounds of Section \ref{sec2.1}. In the Supplementary Material (SM), we perform a computation from Example \ref{ex3.5}, prove Lemma \ref{cbwbshc}, and provide detailed calculations to support claims made in the proofs of Theorems \ref{thm2.1}--\ref{thm2.2univ}.

\vspace{2mm}

\noindent{\emph{Notation.}} The class $C^t_b(I)$ consists of all functions $h:I\subset\mathbb{R}\rightarrow\mathbb{R}$ for which $h^{(t-1)}$ exists and is absolutely continuous and has bounded derivatives up to the $t$-th order. For a given function $P:\mathbb{R}^d\rightarrow \mathbb{R}^+$, the class $C_{P}^{t,m}(\mathbb{R}^d)$ consists of all functions $\mathbf{g}:\mathbb{R}^d\rightarrow\mathbb{R}^m$ such that all $t$-th order partial derivatives of the $m$ components of $\mathbf{g}=(g_1,\ldots,g_m)^\intercal$ exist and are such that, for all $\mathbf{w}\in\mathbb{R}^d$,
\begin{equation*}\bigg|\frac{\partial^kg_l(\mathbf{w})}{\prod_{j=1}^k\partial w_{i_j}}\bigg|^{t/k}\leq P(\mathbf{w}), \quad k=1,\ldots,t, \: l=1,\ldots,m.
\end{equation*}
For given functions $P=(P_0,P_1,\ldots,P_n)$, 
the class $\mathcal{C}_P^{t,n,m}(\mathbb{R}^d)$ consists of all functions $\mathbf{f}:\mathbb{R}^d\rightarrow\mathbb{R}^m$ whose $(t+n)$-th order partial derivatives of each of the components $f_1,\ldots,f_m$ exist, and are such that, for all $\mathbf{w}\in\mathbb{R}^d$, 
\begin{equation*}\bigg|\frac{\partial^{t+k}f_l(\mathbf{w})}{\prod_{j=1}^{t+k}\partial w_{i_j}}\bigg|\leq P_k(\mathbf{w}), \quad k=0,1,\ldots,n, \: l=1,\ldots,m.
\end{equation*}
 We also denote $h_{p,m}=\sum_{k=1}^pm^k{p\brace k}|h|_k$, where $|h|_k=\sup_{i_1,\ldots,i_k}\big\|\frac{\partial^kh}{\prod_{j=1}^k\partial w_{i_j}}\big\|$, and ${n\brace k}=(1/k!)\sum_{j=0}^k(-1)^{k-j}\binom{k}{j}j^n$ is a Stirling number of the second kind (see \cite{olver}). We let $h_{p}:=h_{p,1}$. We will sometimes use the notation $\partial_x=\frac{\partial}{\partial x}$, $\partial_{xy}=\frac{\partial^2}{\partial x\partial y}$ and $\partial_x^2=\frac{\partial}{\partial x^2}$ for lower-order partial derivatives. For a function $h$ and random elements $X$ and $Y$, we let $\Delta_h(X,Y)$ denote the quantity $|\mathbb{E}[h(X)]-\mathbb{E}[h(Y)]|$. We denote that the $r$-th absolute moment of the $\mathrm{N}(0,\sigma^2)$ distribution by $\mu_{r,\sigma}=2^{r/2}\sigma^r\Gamma((r+1)/2)/\sqrt{\pi}$. For $a,b\in\mathbb{R}$, we write $a\vee b:=\max\{a,b\}$. We use the convention that $0^0:=1$.

\section{Main results}\label{sec2}
		
\subsection{General bounds for the delta method}\label{sec2.1}


Let  $\mathbf{X}_1,\ldots, \mathbf{X}_n$ be independent zero mean random vectors, with $\mathbf{X}_{i}=(X_{i,1},\ldots,X_{i,d})^\intercal$, $i=1,\ldots,n$. Denote the sample mean by $\overline{\mathbf{X}}=n^{-1}\sum_{i=1}^n\mathbf{X}_i$. For $j=1,\ldots,d$, let $W_j=n^{-1/2}\sum_{i=1}^nX_{ij}$ and denote $\mathbf{W}=(W_1,\ldots,W_d)^\intercal$. Suppose that the covariance matrix $\Sigma\in\mathbb{R}^{d\times d}$ of $\mathbf{W}$ is non-negative definite (so that the covariance matrix $n^{-1}\Sigma$ of $\overline{\mathbf{X}}$ is also non-negative definite) and let $\mathbf{Z}\sim\mathrm{N}_d(\mathbf{0},\Sigma)$. Let $\sigma_{jk}=(\Sigma)_{jk}$ and set $\sigma_j^2=\sigma_{jj}$.
Assume that the first $t+1$ partial derivatives of the function $\mathbf{f}:\mathbb{R}^d\rightarrow\mathbb{R}^m$ exist and are such that
 $\frac{\partial^{t}\mathbf{f}}{\prod_{j=1}^{t}\partial w_{i_j}}(\mathbf{0})\neq\mathbf{0}$  for some $1\leq i_1,\ldots,i_t\leq d$, with (in the case $t\geq2$) all lower order partial derivatives of $\mathbf{f}$ evaluated at the mean $\mathbf{0}$ being equal to the zero vector $\mathbf{0}$. Recall that we define  $\mathbf{T}_{t,n}=n^{t/2}(\mathbf{f}(\overline{\mathbf{X}})-\mathbf{f}(\mathbf{0}))$, and that the limit random vector $\mathbf{Y}_t$ is defined as in (\ref{limity}). We shall assume that the derivatives of the components of $\mathbf{f}$ have polynomial growth rate. To this end, we shall suppose $\mathbf{f}\in\mathcal{C}_P^{t,n,m}(\mathbb{R}^d)$ with $P=(P_0,P_1,\ldots,P_n)$, where $P_k(\mathbf{w})=A_{t+k}(1+\sum_{i=1}^d|w_i|^{r_t+k})$, $k=0,1,\ldots,n$. Let $a_{n,d,r_t}= d/n^{r_t/2}\vee1$.

\begin{theorem}\label{thm2.1}Let the random vectors $\mathbf{X}_1,\ldots,\mathbf{X}_n$ and the function $\mathbf{f}:\mathbb{R}^d\rightarrow\mathbb{R}^m$ be defined as above. In addition, suppose that $\mathbf{f}\in \mathcal{C}_P^{1,2,m}(\mathbb{R}^d)$ for $t=1$ and $\mathbf{f}\in \mathcal{C}_P^{t,1,m}(\mathbb{R}^d)$ for $t\geq2$. Suppose also that  $\mathbb{E}|X_{ij}|^{u_{1,t}+3}<\infty$ for all $i,j$, where $u_{1,1}=\max\{3r_1,3r_{2}/2,r_{3}\}$, $u_{1,2}=\max\{3(r_2+1),r_3\}$ and $u_{1,t}=3(r_t+t-1)$ for $t\geq3$. Let $C_{1,1}=\max\{4A_1^3,\sqrt{2}A_2^{3/2}/\sqrt{n},A_3/n^{5/6}\}$, $C_{1,2}=\max\{4A_2^3,\sqrt{2}A_2^{3/2},A_3/\sqrt{n}\}$ and $C_{1,t}=\max\{\frac{4A_t^{3}}{((t-1)!)^3}, \frac{\sqrt{2}A_t^{3/2}}{((t-2)!)^{3/2}}, \frac{A_t}{(t-3)!}\}$
for $t\geq3$. 
Suppose $n\geq d^6\vee 8$. 
Then, for $h\in C_b^3(\mathbb{R}^m)$, 
\begin{align}\label{thm2.1bd}
\Delta_h(\mathbf{T}_{t,n},\mathbf{Y}_t)\leq\frac{1}{\sqrt{n}}\big(m|h|_1\mathcal{M}_{1,d}+h_{3,m}\mathcal{M}_{2,d}\big),
\end{align}
where
\begin{align*}
	\mathcal{M}_{1,d}&=\frac{A_{t+1}d^t}{(t+1)!}\sum_{j=1}^{d}\bigg(\mu_{t+1,\sigma_j}
	+\frac{d}{n^{r_{t+1}/2}}\mu_{r_{t+1}+t+1,\sigma_j}\bigg),\\
	\mathcal{M}_{2,d}&=\frac{C_{1,t}a_{n,d,r_t}^3 d^{3t-2}}{n}\sum_{i=1}^n\sum_{j,k=1}^d\big((1+2^{u_{1,t}/2}\mu_{u_{1,t},\sigma_k})\mathbb{E}|X_{ij}|^3+\mathbb{E}|X_{ij}|^{u_{1,t}+3}\nonumber\\
	&\quad+2^{3u_{1,t}/2}\mathbb{E}|X_{ij}|^3\mathbb{E}|W_k|^{u_{1,t}}\big).
	\end{align*}
\end{theorem}

\begin{theorem}\label{thm2.2}Let $\mathbf{X}_1,\ldots,\mathbf{X}_n$ and $\mathbf{f}:\mathbb{R}^d\rightarrow\mathbb{R}^m$ be defined as above. In addition, suppose $t\geq2$ is even, that $\mathbf{f}$ is an even function (that is $\mathbf{f}(\mathbf{x})=\mathbf{f}(-\mathbf{x})$ for all $\mathbf{x}\in\mathbb{R}^d$), and that $\mathbf{f}\in \mathcal{C}_P^{2,4,m}(\mathbb{R}^d)$ for $t=2$ and $\mathbf{f}\in \mathcal{C}_P^{t,2,m}(\mathbb{R}^d)$ for $t\geq4$. Suppose also that $\mathbb{E}|X_{ij}|^{u_{2,t}+4}<\infty$ for all $i,j$, where $u_{2,2}=\max\{6(r_2+1), 2r_3, 3/2r_4, 6/5r_5, r_6\}$, $u_{2,4}=\max\{6(r_4+3),6r_5/5,r_6\}$ and $u_{2,t}=6(r_t+t-1)$, $t\geq6.$ Let
\begin{align*}{C}_{2,2}&=\max\{32A_2^6, 4A_2^3, \tfrac{4A_3^2}{n}, \tfrac{\sqrt{2}A_4^{3/2}}{n^{3/2}}, \tfrac{2^{1/5}A_5^{6/5}}{n^{9/5}}, \tfrac{A_6}{n^2}\}, \\
{C}_{2,4}&=\max\{\tfrac{A_4^6}{1458},\tfrac{A_4^3}{2},2A_4^2,\sqrt{2}A_4^{3/2},\tfrac{2^{1/5}A_5^{6/5}}{n^{3/5}},\tfrac{A_6}{n}\}, \\
{C}_{2,t}&=\max\big\{\tfrac{32A_t^6}{((t-1)!)^6},\tfrac{4A_t^3}{((t-2)!)^3},\tfrac{2A_t^2}{((t-3)!)^2},\tfrac{\sqrt{2}A_t^{3/2}}{((t-4)!)^{3/2}},\tfrac{2^{1/5}A_t^{6/5}}{((t-5)!)^{6/5}},\tfrac{A_t}{(t-6)!}\big\}, \quad t\geq6.
\end{align*} 
Suppose $n\geq12$.
Then, for $h\in C_b^6(\mathbb{R}^m)$, 
\begin{align}
\Delta_h(\mathbf{T}_{t,n},\mathbf{Y}_t)\leq\frac{1}{n}\big(m|h|_1\mathcal{K}_{1,d}+m^2|h|_2\mathcal{K}_{2,d}+h_{4,m}\mathcal{K}_{3,d}+h_{6,m}\mathcal{K}_{4,d}\big),\label{thm2.2bd}
\end{align}
where
\begin{align*}
	\mathcal{K}_{1,d}&=\frac{A_{t+2}d^{t+1}}{(t+2)!}\sum_{j=1}^{d}\bigg(\mu_{t+2,\sigma_j}+\frac{d}{n^{r_{t+2}/2}}\mu_{r_{t+2}+t+2,\sigma_j}\bigg),\\
	\mathcal{K}_{2,d}&=\frac{d^{2t+1}}{((t+1)!)^2}\sum_{j=1}^{d}\bigg(A_{t+1}^2\mu_{2(t+1),\sigma_j}+\frac{2A_{t+2}^2d^2}{(t+2)^2n}\big(\mu_{2(t+2),\sigma_j}+d^2\mu_{2(r_{t+2}+t+2),\sigma_j}\big)\bigg),\\
\mathcal{K}_{3,d}&=\frac{13{C}_{2,t}a_{n,d,r_t}^6d^{6t-4}}{12n}\sum_{i=1}^n\sum_{j,k=1}^d\big((1+2^{u_{2,t}/2}\mu_{u_{2,t},\sigma_k})\mathbb{E}[X_{ij}^4]+\mathbb{E}|X_{ij}|^{u_{2,t}+4}\\
&\quad+2^{3u_{2,t}/2}\mathbb{E}[X_{ij}^4]\mathbb{E}|W_k|^{u_{2,t}}\big),\\
\mathcal{K}_{4,d}&=\frac{{C}_{2,t}a_{n,d,r_t}^6d^{6t-5}}{12n^2}\sum_{i,\alpha=1}^n\sum_{j,k,l,a,q=1}^d|\mathbb{E}[X_{i j}X_{i k}X_{i l}]|\big((1+ 2\cdot3^{u_{2,t}/2}\mu_{u_{2,t},\sigma_q})\mathbb{E}|X_{\alpha a}|^3\\
&\quad+\mathbb{E}|X_{\alpha a}|^{u_{2,t}+3}+12^{u_{2,t}/2}\mathbb{E}|X_{\alpha a}|^3\mathbb{E}|W_q|^{u_{2,t}}\big).
\end{align*}
\end{theorem}

\begin{theorem}\label{thm2.3}Let $\mathbf{X}_1,\ldots,\mathbf{X}_n$ and $\mathbf{f}:\mathbb{R}^d\rightarrow\mathbb{R}^m$ (not necessarily an even function) be defined as above.  Suppose further that $t\geq2$ is even and $\mathbf{f}\in \mathcal{C}_P^{2,2,m}(\mathbb{R}^d)$ for $t=2$ and that $\mathbf{f}\in \mathcal{C}_P^{t,2,m}(\mathbb{R}^d)$ for $t\geq4$. Suppose also that $\mathbb{E}[X_{ij}X_{ik}X_{il}]=0$, for all $i,j,k,l$, and that $\mathbb{E}|X_{ij}|^{u_{3,t}+4}<\infty$ for all $i,j$, where $u_{3,2}=\max\{4(r_2+1),4r_3/3,r_4\}$ and $u_{3,t}=4(r_t+t-1)$, $t\geq4$. Let
 $C_{3,2}=\max\{8A_2^4, {2}A_2^{2}, 2^{1/3}A_3^{4/3}/n^{2/3}, A_4/n\}$ and $C_{3,t}=\max\{\frac{8A_t^4}{((t-1)!)^4},\frac{2A_t^{2}}{((t-2)!)^2},\frac{2^{1/3}A_t^{4/3}}{((t-3)!)^{4/3}}, \frac{A_t}{(t-4)!}\}$ for $t\geq4$. Suppose $n\geq8$.
 Then, for $h\in C_b^4(\mathbb{R}^m)$, 
\begin{align}\Delta_h(\mathbf{T}_{t,n},\mathbf{Y}_t)\leq\frac{1}{n}\big(m|h|_1\mathcal{K}_{1,d}+m|h|_2\mathcal{K}_{2,d}+h_{4,m}\mathcal{K}_{5,d}\big),\label{thm2.3bound}
\end{align}
where $\mathcal{K}_{1,d}$ and $\mathcal{K}_{2,d}$ are defined in the statement of Theorem \ref{thm2.2} and 
\begin{align*}
	\mathcal{K}_{5,d}&=\frac{5{C}_{3,t}a_{n,d,r_t}^4d^{4t-2}}{6n}\sum_{i=1}^n\sum_{j,k=1}^d\big((1+2^{u_{3,t}/2}\mu_{u_{3,t},\sigma_k})\mathbb{E}[X_{ij}^4]+\mathbb{E}|X_{ij}|^{u_{3,t}+4}\\
&\quad+2^{3u_{3,t}/2}\mathbb{E}[X_{ij}^4]\mathbb{E}|W_k|^{u_{3,t}}\big).
\end{align*}
\end{theorem}

The following theorems provide improved bounds in the univariate delta method ($d=m=1$). The bounds hold for wider classes of test functions $h$, and require weaker moment conditions and weaker assumptions on the derivatives of the function $f:\mathbb{R}\rightarrow\mathbb{R}$. We simplify notation as follows. For zero mean independent random variables $X_1,\ldots,X_n$, we denote their sample mean by $\overline{X}=n^{-1}\sum_{i=1}^nX_i$, and also let $W=n^{-1/2}\sum_{i=1}^nX_i$. We suppose that $\mathrm{Var}(W)=\sigma^2\in(0,\infty)$. The condition on the derivatives of $f:\mathbb{R}\rightarrow\mathbb{R}$ at the mean $0$ now simply reads that $f^{(t)}(0)\not=0$ and (in the case $t\geq2$) $f^{(j)}(0)=0$ for $j=1,\ldots, t-1$.

\begin{theorem}\label{thm2.1univ} Let the random variables $X_1,\ldots,X_n$ and the function $f:\mathbb{R}\rightarrow\mathbb{R}$ be defined as above. Suppose that, for $t\geq1$, $f\in \mathcal{C}^{t,1,1}_P(\mathbb{R})$ and $\mathbb{E}|X_i|^{u_{4,t}+3}<\infty$ for all $i$, where $u_{4,t}=r_t+t-1$. Suppose $n\geq8$. Then, for $h\in C^1_b(\mathbb{R}),$
	\begin{align}
		\Delta_h(T_{t,n},Y_t)\leq\frac{\|h'\|}{\sqrt{n}}(\mathcal{M}_{1,1}+\mathcal{M}_3),
	\label{thm2.4bd}	
	\end{align}
where $\mathcal{M}_{1,1}$ is defined as in the statement of Theorem \ref{thm2.1} and
\begin{align*}\mathcal{M}_3&=\frac{3A_t}{(t-1)!\sigma^2n}\sum_{i=1}^{n}\big(
	(\alpha_{u_{4,t}}+2^{u_{4,t}/2}\gamma_{u_{4,t},\sigma})\mathbb{E}|X_i|^3\nonumber\\
&\quad+2^{3u_{4,t}/2}\beta_{u_{4,t}}\mathbb{E}|X_i|^3\mathbb{E}|W|^{u_{4,t}}+\beta_{u_{4,t}}\mathbb{E}|X_i|^{u_{4,t}+3}\big),
\end{align*}	
with $(\alpha_r,\beta_r,\gamma_{r,\sigma})=(4,4,2\mu_{r,\sigma})$ if $0\leq r\leq1$, and	$(\alpha_r,\beta_r,\gamma_{r,\sigma})=(r+3,r+5,\sigma^{-1}(r+1)\mu_{r+1,\sigma})$ if $r>1$.
\end{theorem}

\begin{theorem}\label{thm2.2univ} Let $X_1,\ldots,X_n$ and $f:\mathbb{R}\rightarrow\mathbb{R}$ be defined as above. Suppose $t\geq2$ is even and that $f$ is an even function. Suppose further that $f\in\mathcal{C}^{t,2,1}_P(\mathbb{R})$ and $\mathbb{E}|X_i|^{u_{5,t}+4}<\infty$ for all $i$, where $u_{5,t}=2(r_t+t-1)$. Let ${C}_{4,t}=\max\{A_t/(t-2)!,2A_t^2/((t-1)!)^2\}$. Suppose $n\geq12$. Then, for $h\in C^2_b(\mathbb{R})$,
\begin{align}\label{thm2.5bd}\Delta_h(T_{t,n},Y_t)\leq\frac{1}{n}\Big(\|h'\|\mathcal{K}_{1,1}+\|h''\|\mathcal{K}_{2,1}+(\|h'\|+\|h''\|)\Big(\frac{13}{10}\mathcal{K}_{6}+\mathcal{K}_7\Big)\Big),
\end{align}
where $\mathcal{K}_{1,1}$ and $\mathcal{K}_{2,1}$ are defined as in the statement of Theorem \ref{thm2.2} and
\begin{align*}
		\mathcal{K}_6&=\frac{10C_{4,t}}{3\sigma^2n}\sum_{i=1}^{n}\big(
	(\alpha_{u_{5,t}}+2^{u_{5,t}/2}\gamma_{u_{5,t},\sigma})\mathbb{E}[X_i^4]+2^{3u_{5,t}/2}\beta_{u_{5,t}}\mathbb{E}[X_i^4]\mathbb{E}|W|^{u_{5,t}}\\
	&\quad+\beta_{u_{5,t}}\mathbb{E}|X_i|^{u_{5,t}+4}\big),\nonumber\\
	\mathcal{K}_7&=\frac{3C_{4,t}}{2\sigma^4n^2}\sum_{i,\alpha=1}^{n}|\mathbb{E}[X_i^3]|\big(
	(\tilde\alpha_{u_{5,t}}+3^{u_{5,t}/2}\tilde\gamma_{u_{5,t},\sigma})\mathbb{E}|X_\alpha|^3+12^{u_{5,t}/2}\tilde\beta_{u_{5,t}}\mathbb{E}|X_\alpha|^3\mathbb{E}|W|^{u_{5,t}}\\
	&\quad+\tilde\beta_{u_{5,t}}\mathbb{E}|X_\alpha|^{u_{5,t}+3}\big),
\end{align*}
where $(\tilde\alpha_r,\tilde\beta_r,\tilde\gamma_{r,\sigma})=(10,10,10\mu_{r+1,\sigma})$ if $0\leq r\leq1$, and $(\tilde\alpha_r,\tilde\beta_r,\tilde\gamma_{r,\sigma})=(r^2+r+8,r^2+2r+18,\sigma^{-1}(2r^2+r+5)\mu_{r+1,\sigma})$ if $r>1$.
\end{theorem}

\begin{theorem}\label{thm2.3univ}Let $X_1,\ldots,X_n$ and $f:\mathbb{R}\rightarrow\mathbb{R}$ (not necessarily an even function) be defined as above. Suppose further that $t\geq2$ is even and that $f\in\mathcal{C}^{t,2,1}_P(\mathbb{R})$. Suppose $\mathbb{E}[X_i^3]=0$ for all $1\leq i\leq n$, and $\mathbb{E}|X_i|^{u_{5,t}+4}<\infty$ for all $i$. Suppose $n\geq8$. Then, for $h\in C^2_b(\mathbb{R})$,
	\begin{align}\label{thm2.6bd}
	\Delta_h(T_{t,n},Y_t)\leq \frac{1}{n}\big(\|h'\|\mathcal{K}_{1,1}+\|h''\|\mathcal{K}_{2,1}+(\|h'\|+\|h''\|)\mathcal{K}_6\big).
	\end{align}
\end{theorem}

\begin{remark}If one only requires bounds with an explicit dependence on the sample size $n$ and dimensions $d$ and $m$, but not explicit constants, then it is possible to obtain simplified bounds under the assumptions of each of Theorems \ref{thm2.1}--\ref{thm2.3univ}, as well as Theorems \ref{theoremsec21}--\ref{indep2} in the next section. For example, under the assumptions of Theorem \ref{thm2.1}, there exists a constant $C$ that does not depend on $n$, $d$ and $m$ such that
\begin{align*}\Delta_h(\mathbf{T}_{t,n},\mathbf{Y}_t)\leq\frac{Ca_{n,d,r_t}^3d^{3t-1}h_{3,m}}{n^{3/2}}\sum_{i=1}^n\sum_{j=1}^d(1+\mathbb{E}|X_{ij}|^{u_{1,t}+3}).
\end{align*}
This simplification is obtained from the following considerations. Firstly, we have the bound $\mathbb{E}|X_{ij}|^3\leq 1+\mathbb{E}|X_{ij}|^{u_{1,t}+3}$, due to the basic inequality $|x|^a\leq 1+|x|^{b}$ for $0\leq a<b$. Secondly, for a constant $K$ not depending on $n$, we have $\mathbb{E}|W_k|^r\leq K(1+n^{-1}\sum_{i=1}^{n}\mathbb{E}|X_{ik}|^{r})$.
For $r\geq2$, this follows from the Marcinkiewiczi-Zygmund inequality \cite{mz37} $\mathbb{E}|\sum_{i=1}^{n}Y_i|^r\leq C_r\{\sum_{i=1}^{n}\mathbb{E}|Y_i|^r+(\sum_{i=1}^{n}\mathbb{E}[Y_i^2])^{1/2}\}$ and the inequality $\mathbb{E}|Y|^a\leq1+\mathbb{E}|Y|^b,$ for any $0\leq a< b$. For $0\leq r<2$, we simply apply H\"older's inequality, $\mathbb{E}|W_k|^r\leq(\mathbb{E}[W_k^2])^{r/2}=\sigma_k^r$.  Thirdly, by H\"{o}lder's inequality, we have $\mathbb{E}|X_{ij}|^a\mathbb{E}|X_{ik}|^b\leq(\mathbb{E}|X_{ij}|^{a+b})^{a/(a+b)} (\mathbb{E}|X_{ik}|^{a+b})^{b/(a+b)} \leq\max\{\mathbb{E}|X_{ij}|^{a+b},\mathbb{E}|X_{ik}|^{a+b}\}$, for $a,b\geq1$.
\end{remark}

\begin{remark}\label{vfnd}In Theorem \ref{thm2.1}, we assumed that $n\geq d^6$ in order to obtain a more compact final bound. This assumption is very mild, as we require that $n/ (d^{6t}m^6)\rightarrow\infty$ for the bound (\ref{thm2.1bd}) of Theorem \ref{thm2.1} to tend to zero. 

If $r_t\geq1/(3t)$ (so that $a_{n,d,r_t}=1$ for $n\geq d^{6t}$), the bound (\ref{thm2.1bd}) tends to zero if  $n/ (d^{6t}m^6)\rightarrow\infty$. 
If $0\leq r_t<1/(3t)$, then the dependence on the dimension becomes worse. For example, when $r_t=0$, we have $a_{n,d,0}=d$, and for the bound (\ref{thm2.1bd}) to tend to zero we require that $n/(d^{6t+6}m^6)\rightarrow\infty$. In most applications (this is the case for all examples from Section \ref{sec3}), we would have either $r_t>1/(3t)$ or $r_t=0$. In the latter case, we can apply the bounds of Section \ref{sec2.2}. Moreover, if we can find $A_t>0$ and $0\leq r_t<1/(3t)$ such that $|\frac{\partial^tf_l(\mathbf{w})}{\prod_{j=1}^t\partial w_{i_j}}|\leq A_t(1+\sum_{i=1}^d|w_i|^{r_t})$ for all $\mathbf{w}\in\mathbb{R}^d$, then we can also find $A_t'>0$ and $ r_t'\geq1/(3t)$ such that $|\frac{\partial^tf_l(\mathbf{w})}{\prod_{j=1}^t\partial w_{i_j}}|\leq A_t'(1+\sum_{i=1}^d|w_i|^{r_t'})$ for all $\mathbf{w}\in\mathbb{R}^d$, since $|w_i|^{r_t}\leq 1+|w_i|^{r_t'}$. Therefore, if for the problem at hand, the dependence on the dimension $d$ is more important than the moment conditions, we can take $r_t\geq1/(3t)$.  Similar comments apply to the bounds of Theorems \ref{thm2.2} and \ref{thm2.3}; in particular, the bounds in Theorems \ref{thm2.2} and \ref{thm2.3} tend to zero even when the dimensions $d$ and $m$ are allowed to grow with $n$ provided $n/ (d^{6t}m^6)\rightarrow\infty$ (if $r_t\geq1/(3t)$, so that $a_{n,d,r_t}=1$ for $n\geq d^{6t}$) and $n/ (d^{4t}m^4)\rightarrow\infty$ (if $r_t\geq1/(2t)$, so that $a_{n,d,r_t}=1$ for $n\geq d^{4t}$), respectively.  

\end{remark}

\begin{remark}\label{rem2.1} The premise of smooth test functions is crucial, because $O(n^{-1})$ bounds like those of Theorems \ref{thm2.2}, \ref{thm2.3}, \ref{thm2.2univ} and \ref{thm2.3univ} will in general not hold for non-smooth test functions. To see this, consider the single point test function $h\equiv\chi_{\{0\}}$ and suppose $n=2k$. Consider the setting of the aforementioned theorems with $f(x)=x^t$, $T_{t,n}=n^{t/2}(f(\overline{X})-f(0))=n^{t/2}\overline{X}^t$, $Y_t=Z^t$ (where $Z\sim\mathrm{N}(0,1)$), and let $X_1,\ldots,X_n$ be i.i.d.\ with $\mathbb{P}(X_1=-1)=\mathbb{P}(X_1=1)=1/2.$
	Then
 \[\mathbb{E}\big[h\big(n^{t/2}\overline{X}^t\big)\big]=\mathbb{P}\big(n^{t/2}\overline{X}^t=0\big)=\mathbb{P}\Big(\sum_iX_i=0\Big)=\binom{2k}{k}\Big(\frac{1}{2}\Big)^{2k}\approx\frac{1}{\sqrt{\pi k}}=\sqrt{\frac{2}{\pi n}},\]
 where we used Stirling's approximation. Since $\mathbb{E}[h(Z^t)]=\mathbb{P}(Z^t=0)=0$, the Kolmogorov distance between the distributions of $n^{t/2}\overline{X}^t$ and $Z^t$ cannot be of smaller order than $n^{-1/2}$.
\end{remark}

\begin{remark}\label{rem2.2} The bounds of Theorems \ref{thm2.1}--\ref{thm2.3univ} can be expressed in terms of convergence determining integral probability metrics (IPMs) as follows. 
Let
\begin{align*}\mathcal{H}_{p}&=\{h:\mathbb{R}^m\rightarrow\mathbb{R}\,:\,\text{$h^{(p-1)}$ is Lipschitz with $|h|_j\leq1$, $1\leq j\leq p$}\}.
\end{align*} 
Then we define the IPM between the laws of $\mathbb{R}^m$-valued random vectors $\mathbf{X}$ and $\mathbf{Y}$ by $d_{p}(\mathbf{X},\mathbf{Y})=\sup_{h\in\mathcal{H}_{p}}|\mathbb{E}[h(\mathbf{X})]-\mathbb{E}[h(\mathbf{Y})]|$ (with standard abuse of notation). Note that $d_1$ corresponds to the Wasserstein distance, which we denote by $d_{\mathrm{W}}$. In this case, we obtain from the bound (\ref{thm2.4bd}) of Theorem \ref{thm2.1univ} the Wasserstein distance bound 
\begin{equation}\label{dw1}d_{\mathrm{W}}(T_{t,n},Y_t)\leq(\mathcal{M}_{1,1}+\mathcal{M}_3)/\sqrt{n}.
\end{equation} 
Similarly, setting $|h|_k=1$, $1\leq k\leq 6$, in the bounds of Theorems \ref{thm2.1}--\ref{thm2.3}, so that $h_{3,m}=\sum_{k=1}^3m^k{3\brace k}=m+3m^2+m^3$, $h_{4,m}=m+7m^2+6m^3+m^4$ and $h_{6,m}=m+31m^2+90m^3+65m^4+15m^5+m^6$, yields bounds in the $d_3$, $d_6$ and $d_4$ metrics, respectively. 

The Kolmogorov distance is an IPM induced by the class of test functions $\mathcal{H}_{\mathrm{K}}=\{\mathbf{1}_{\cdot\leq z}\,:\,z\in\mathbb{R}^m\}$. We can extract (typically sub-optimal order) Kolmogorov distance bounds from the bounds of Theorems \ref{thm2.1}--\ref{thm2.3univ} by applying bounds of \cite{gl22}. Consider first the univariate case. For $t\geq1$, recall that $Y_t=_df^{(t)}(0)Z^t/t!$, where $Z\sim \mathrm{N}(0,\sigma^2)$. Denote the density of $Y_t$ by $p_t$. Then one can readily check that there exists a constant $C>0$ such that $p_t(x)\leq C|x|^{1/t-1}$, for all $x\in\mathbb{R}$ (see, for example, \cite[p.\ 3112]{gaunt pn}). Thus, applying part (iii) of Proposition 2.1 of \cite{gl22} tells us that there exist constants $C'$ and $C''$ such that
\begin{align}\label{dk1}d_{\mathrm{K}}(T_{t,n},Y_t)&\leq C'\big(d_{\mathrm{W}}(T_{t,n},Y_t)\big)^{1/(1+t)}, \\ 
\label{dk2}d_{\mathrm{K}}(T_{t,n},Y_t)&\leq C''\big(d_{2}(T_{t,n},Y_t)\big)^{1/(1+2t)}
\end{align}
(with a more careful analysis, explicit constants can be given in place of $C'$ and $C''$). We observe that substituting the $O(n^{-1/2})$ bound (\ref{dw1}) into inequality (\ref{dk1}) yields a Kolmogorov distance bound of order $n^{-1/(2+2t)}$, whilst inserting the $O(n^{-1})$ bounds of Theorem \ref{thm2.2univ} and \ref{thm2.3univ} (with $\|h'\|=\|h''\|=1$) into inequality (\ref{dk2}) yields a Kolmogorov distance bound of order $n^{-1/(1+2t)}$.

In the general multivariate case with $t\geq1$ and $d>1$, we are not aware of bounds for the density of the limit random vector $\mathbf{Y}_t$, so at that level of generality we are unable to apply the bounds of \cite{gl22} to deduce Kolmogorov distance bounds. However, in the case $t=1$, corresponding to multivariate normal approximation, we can apply Proposition 2.6 of \cite{gl22} to obtain the following bound. Let $\sigma^2=\min_{1\leq j\leq d}\sigma_j^2$. Then, provided $d_3(\mathbf{T}_{1,n},\mathbf{Y}_1)\leq (2+\sqrt{2\log d})/(2\sigma^2)$, we have
\begin{align}\label{dk3}d_{\mathrm{K}}(\mathbf{T}_{1,n},\mathbf{Y}_1)\leq6.17\bigg(\frac{\sqrt{\log d}+\sqrt{2}}{\sigma^2}\bigg)^{3/4}\big(d_3(\mathbf{T}_{1,n},\mathbf{Y}_1)\big)^{1/4}.
\end{align}
Applying the $O(n^{-1/2})$ bound of Theorem \ref{thm2.1} to inequality (\ref{dk3}) yields a Kolmogorov distance bound of order $n^{-1/8}$. We observe that this Kolmogorov distance bound will tend to zero if $(\log d)^{3}d^{12}/n\rightarrow0$, which is the same condition up to a logarithmic factor under which the  bound of Theorem \ref{thm2.1} tends to zero when $t=1$.
\end{remark}

\subsection{The case of vanishing $(t+1)$-th order partial derivatives}\label{sec2.2}

Let $\mathbf{g}_{t,n}(\mathbf{w})=n^{t/2}(\mathbf{f}(n^{-1/2}\mathbf{w})-\mathbf{f}(\mathbf{0}))$ and let $\mathbf{W}=\sqrt{n}\overline{\mathbf{X}}$, so we have $\mathbf{T}_{t,n}=\mathbf{g}_{t,n}(\mathbf{W}).$ Also, let $\mathbf{Z}\sim\mathrm{N}_d(\mathbf{0},\Sigma)$. Suppose that all $(t+1)$-th order partial derivatives of $\mathbf{f}$ are equal to the zero vector $\mathbf{0}$,  in addition to the standing assumption from Section \ref{sec2.1} that  $\frac{\partial^{t}\mathbf{f}}{\prod_{j=1}^{t}\partial w_{i_j}}(\mathbf{0})\neq\mathbf{0}$  for some $1\leq i_1,\ldots,i_t\leq d$, with (in the case $t\geq2$) all lower order partial derivatives of $\mathbf{f}$ evaluated at the mean $\mathbf{0}$ being equal to the zero vector $\mathbf{0}$. In this case, 
$\mathbf{g}_{t,n}(\mathbf{Z})=_d\mathbf{Y}_t$ (this can be read off from the proof of Theorem \ref{thm2.1}). Therefore
	\begin{equation}\label{rhs}
		\Delta_h(\mathbf{T}_{t,n},\mathbf{Y}_t)=
		\Delta_h(\mathbf{g}_{t,n}(\mathbf{W}),\mathbf{g}_{t,n}(\mathbf{Z})),
	\end{equation}
and the quantity on the right-hand side of (\ref{rhs}) can  be bounded by applying the bounds of Theorems \ref{theoremsec21}--\ref{indep2}, below. The bounds of these theorems are used in our proofs of Theorems \ref{thm2.1}--\ref{thm2.3univ}. 
These bounds are also widely applicable and are of independent interest. 
For instance, in the multivariate case ($d,m\geq1$), an application of the bounds in this section may yield bounds with an improved dependence on the dimension $d$ than would result from applying the bounds of Section \ref{sec2.1}; see Example \ref{ex3.5} for an illustration.  


As in Section \ref{sec2.1}, we let  $\mathbf{X}_1,\ldots, \mathbf{X}_n$ be independent zero mean random vectors, with $\mathbf{X}_{i}=(X_{i,1},\ldots,X_{i,d})^\intercal$, $i=1,\ldots,n$. 
For $j=1,\ldots,d$, we let $W_j=n^{-1/2}\sum_{i=1}^nX_{ij}$ and denote $\mathbf{W}=(W_1,\ldots,W_d)^\intercal$.   Suppose that the covariance matrix $\Sigma$ of $\mathbf{W}$ is non-negative definite.  
We will assume that the derivatives of the function $\mathbf{g}:\mathbb{R}^d\rightarrow\mathbb{R}^m$ have polynomial growth. To this end, we suppose $\mathbf{g}\in C_{P}^{t,m}(\mathbb{R}^d)$ with $P(\mathbf{w})=A+B\sum_{i=1}^d|w_i|^{r}$, where $A,B\geq 0$ and $r\geq 0$.

The following Theorems \ref{theoremsec21}--\ref{indep2} generalise bounds of Theorem 3.1 of \cite{GauntSut} to vector-valued functions $\mathbf{g}:\mathbb{R}^d\rightarrow\mathbb{R}^m$ and random vectors whose components may be dependent. Theorem \ref{thmsec2} provides simple sufficient conditions under which $O(n^{-1})$ bounds can be obtained for smooth test functions.  
 

\begin{theorem}\label{theoremsec21}
Suppose that the above notations and assumptions prevail. Suppose also that $\mathbb{E}|X_{ij}|^{r+3}<\infty$ for all $i,j$, and that $\mathbf{g}\in C_P^{3,m}(\mathbb{R}^d)$. Suppose $n\geq8$. Then, for $h\in C_b^3(\mathbb{R}^m)$, 
\begin{align}\label{genbd1} \Delta_h(\mathbf{g}(\mathbf{W}),\mathbf{g}(\mathbf{Z}))&\leq \frac{d^2h_{3,m}}{2n^{3/2}}\sum_{i=1}^n\sum_{j,k=1}^d\Big[\Big(\frac{A}{d}+2^{r/2}\mu_{r,\sigma_k}B\Big)\mathbb{E}|X_{ij}|^3+B\mathbb{E}|X_{ij}|^{r+3}\nonumber\\
&\quad+2^{3r/2}B\mathbb{E}|X_{ij}|^3\mathbb{E}|W_k|^r\Big].
\end{align}
\end{theorem}

\begin{theorem}\label{thmsec2}Let $X_{ij}$, $i=1,\ldots,n$, $j=1,\ldots,d$, be defined as in Theorem \ref{theoremsec21}, but with the stronger assumption that $\mathbb{E}|X_{ij}|^{r+4}<\infty$ for all $i,j$. 

\vspace{2mm}

\noindent{(i)} Suppose $\mathbf{g}\in C_P^{6,m}(\mathbb{R}^d)$ is an even function, and that $n\geq12$.  Then, for $h\in C_b^6(\mathbb{R}^m)$,
\begin{align}\label{boundthm2} \Delta_h(\mathbf{g}(\mathbf{W}),\mathbf{g}(\mathbf{Z}))\leq \frac{1}{n}\Big(\frac{13}{10}h_{4,m}K_1+h_{6,m}K_2\Big).
\end{align}

\noindent{(ii)} Suppose $\mathbf{g}\in C_P^{4,m}(\mathbb{R}^d)$ (not necessarily an even function) and that $\mathbb{E}[X_{ij}X_{ik}X_{il}]=0$ for all $1\leq i\leq n$ and all $1\leq j,k,l\leq d$. Suppose $n\geq8$. Then, for $h\in C_b^4(\mathbb{R}^m)$,
\begin{align}\label{boundtheorem} \Delta_h(\mathbf{g}(\mathbf{W}),\mathbf{g}(\mathbf{Z}))\leq \frac{h_{4,m}K_1}{n}.
\end{align}
Here
\begin{align*}K_1&=\frac{5d^3}{12n}\sum_{i=1}^n\sum_{j,k=1}^d\Big[\Big(\frac{A}{d}+2^{r/2}\mu_{r,\sigma_k}B\Big)\mathbb{E}[X_{ij}^4]+B\mathbb{E}|X_{ij}|^{r+4}+2^{3r/2}B\mathbb{E}[X_{ij}^4]\mathbb{E}|W_k|^r\Big], \\
K_2&=\frac{d^2}{24n^2}\sum_{i,\alpha=1}^n\sum_{j,k,l,a,q=1}^d|\mathbb{E}[X_{i j}X_{i k}X_{i l}]|\Big[\Big(\frac{A}{d}+ 2\cdot3^{r/2}\mu_{r,\sigma_q}B\Big)\mathbb{E}|X_{\alpha a}|^3\\
&\quad+B\mathbb{E}|X_{\alpha a}|^{r+3}+12^{r/2}B\mathbb{E}|X_{\alpha a}|^3\mathbb{E}|W_q|^r\Big].
\end{align*}
\end{theorem}

In the univariate case ($d=m=1$) we can obtain bounds that hold under weaker assumptions on the functions $g$ and $h$. We simplify notation, writing $W=n^{-1/2}\sum_{i=1}^nX_i$, where $X_1,\ldots,X_n$ are independent random variables such that $\mathbb{E}[X_i]=0$, $i=1,\ldots,n$, and $\mathbb{E}[W^2]=\sigma^2\in(0,\infty)$. The dominating function also takes the simpler form $P(w)=A+B|w|^r$.
\begin{theorem} \label{indep1}Let $X_1,\ldots,X_n$ be defined as above and suppose $\mathbb{E}|X_i|^{r+3}<\infty$, $1\leq i\leq n$. Suppose also that $g\in C_P^{1,1}(\mathbb{R})$, and $n\geq8$. Then, for $h\in C_b^{1}(\mathbb{R})$,
	\begin{align}
		\Delta_h(g(W),g(Z))&\leq\frac{3\|h'\|}{2\sigma^2n^{3/2}}\sum_{i=1}^{n}\big[
	(A\alpha_r+2^{r/2}B\gamma_{r,\sigma})\mathbb{E}|X_i|^3\nonumber\\
\label{a2bound}	&\quad+2^{3r/2}B\beta_r\mathbb{E}|X_i|^3\mathbb{E}|W|^r+B\beta_r\mathbb{E}|X_i|^{r+3}\big],
	\end{align}
where the constants $\alpha_r,\beta_r,\gamma_{r,\sigma}$, $r\geq0$, are defined as in the statement of Theorem \ref{thm2.1univ}.
\end{theorem}
\begin{theorem}\label{indep2}Let $X_1,\ldots,X_n$ be defined as above and suppose $\mathbb{E}|X_i|^{r+4}<\infty$, $1\leq i\leq n$. Suppose also that $g\in C_P^{2,1}(\mathbb{R})$.

	\vspace{2mm}
	
	\noindent{(i)} Further suppose that $g$ is an even function, and $n\geq12$. Then, for $h\in C^2_b(\mathbb{R})$,
	\begin{align}
			\Delta_h(g(W),g(Z))\leq\frac{1}{n}(\|h'\|+\|h''\|)\Big(\frac{13}{10}K_3+K_4\Big).\label{a.71}
	\end{align}
\noindent{(ii)} Further suppose that $\mathbb{E}[X_i^3]=0$ for all $1\leq i\leq n$, and $n\geq8$ ($g$ is not necessarily an even function). Then, for $h\in C^2_b(\mathbb{R})$,
	\begin{align}
	\Delta_h(g(W),g(Z))&\leq\frac{1}{n}(\|h'\|+\|h''\|)K_3.\label{a.72}
\end{align}
Here, for $\tilde\alpha_r,\tilde\beta_r,\tilde\gamma_{r,\sigma}$, $r\geq0$, are defined as in the statement of Theorem \ref{thm2.2univ},
\begin{align*}
		K_3&=\frac{5}{3\sigma^2n}\sum_{i=1}^{n}\big[
	(A\alpha_r+2^{r/2}B\gamma_{r,\sigma})\mathbb{E}[X_i^4]+2^{3r/2}B\beta_r\mathbb{E}[X_i^4]\mathbb{E}|W|^r+B\beta_r\mathbb{E}|X_i|^{r+4}\big],\nonumber\\
	K_4&=\frac{3}{4\sigma^4n^2}\sum_{i,\alpha=1}^{n}|\mathbb{E}[X_i^3]|\big[
	(A\tilde\alpha_r+3^{r/2}B\tilde\gamma_{r,\sigma})\mathbb{E}|X_\alpha|^3+12^{r/2}B\tilde\beta_r\mathbb{E}|X_\alpha|^3\mathbb{E}|W|^r\\
	&\quad+B\tilde\beta_r\mathbb{E}|X_\alpha|^{r+3}\big].
\end{align*}
\end{theorem}	

\begin{remark}The bound (\ref{genbd1}) of Theorem \ref{theoremsec21} tends to zero if we allow the dimensions $d$ and $m$ to grow with $n$ provided $n/(d^8m^6)\rightarrow\infty$, whilst the bounds (\ref{boundthm2}) and (\ref{boundtheorem}) of Theorem \ref{thmsec2} tend to zero if $n/(d^{7}m^6)\rightarrow\infty$ and $n/(d^{5}m^4)\rightarrow\infty$, respectively.
\end{remark}

\section{Examples}\label{sec3}
In this section, we provide illustrative examples of the application of the general bounds of Section \ref{sec2} in concrete settings. Examples \ref{ex3.1}--\ref{ex3.3} highlight situations in which a change of parameter value can lead to different limit distributions and bounds with a faster $O(n^{-1})$ rate of convergence. Example \ref{ex3.4} provides a simple application of our bounds for vector-valued statistics. In Example \ref{ex3.5}, we obtain optimal order $n^{-1}$ bounds for the chi-square approximation of a family of rank-based statistics introduced by \cite{sen}, and as a special case recover an independent derivation of an order $n^{-1}$ bound for Friedman's statistic. In  Example \ref{ex3.6}, we also provide an independent derivation of an optimal order $n^{-1}$ bound for the chi-square approximation of Pearson's statistic. In order to meet a mild requirement of the theorems from Section \ref{sec2}, we assume that $n\geq12$ in all examples.

\begin{example}[Estimating Bernoulli variance]\label{ex3.1} Consider i.i.d.\ Bernoulli data $X_1,\ldots,X_n$ with parameter $p$. The maximum likelihood estimator (MLE) of $p$ is $\overline{X}=n^{-1}\sum_{i=1}^nX_i$, so the MLE of the variance $p(1-p)$ is given by $\overline{X}(1-\overline{X})$. For $i=1,\ldots,n$, let $V_i=X_i-p$, so that $V_1,\ldots,V_n$ are i.i.d.\ random variables with zero mean. Denote $\overline{V}=n^{-1}\sum_{i=1}^nV_i$. Define $f(v)=(p+v)(1-p-v)$, so that $f(\overline{V})=\overline{X}(1-\overline{X})$.

First, suppose $p\not=1/2$.  We have $f'(0)=1-2p\not=0$, hence we can apply Theorem \ref{thm2.1univ} with $t=1$. Moreover, $|f'(v)|\leq |1-2p|+2|v|\leq 2(1+|v|)$ and $|f''(v)|=2=1+|v|^0$ for all $v\in\mathbb{R}$, so we can take $A_1=2$, $A_2=1$, $r_1=1$ and $r_2=0$. Let $T_{1,n}=\sqrt{n}(\overline{X}(1-\overline{X})-p(1-p))$ and $Y_1\sim \mathrm{N}(0,\sigma_p^2)$, where $\sigma_p^2=p(1-p)(1-2p)^2$. We have that $\mathbb{E}|V_i|^m\leq p(1-p)$ for $m\geq2$, and $\mathbb{E}|W|\leq(\mathbb{E}[W^2])^{1/2}=\sigma=\sqrt{p(1-p)}$ by the Cauchy-Schwarz inequality. We also have that $\mu_{1,\sigma}=\sqrt{2/\pi}\sqrt{p(1-p)}$ and $\mu_{2,\sigma}=p(1-p)$. Now, applying the bound (\ref{thm2.4bd}) yields, $h\in C_b^1(\mathbb{R})$,
\begin{align}\label{addvfdb}\Delta_h(T_{1,n},Y_1) \leq89\|h'\|/\sqrt{n},
\end{align}
where to obtain a compact final bound we used the basic inequality $p(1-p)\leq1/4$ for $0<p<1$, and rounded up the final numerical constant to the nearest integer. Now define $T_{1,n}'=T_{1,n}/\sigma_p$ and $Y_1'\sim N(0,1)$, so that $\mathrm{Var}(T_{1,n}')=1$. Then we immediately deduce from (\ref{addvfdb}) the bound
\begin{align*}\Delta_h(T_{1,n}',Y_1') \leq\frac{89\|h'\|}{\sqrt{np(1-p)(1-2p)^2}},
\end{align*}
which blows up under the expected conditions that $p$ approaches either $0$, $1/2$ or $1$.

Suppose now that $p=1/2$. Then $f(v)=(1/4-v^2)$. We have $f'(0)=0$, $f''(0)=-2$ and $\mathbb{E}[V_i^3]=0$ for all $1\leq i\leq n$, so we can apply Theorem \ref{thm2.3univ} with $t=2$.  Again, we can take $A_2=1$ and $r_2=0$, as well as $A_3=A_4=r_3=r_4=0$ ($f^{(3)}(v)=f^{(4)}(v)=0$). Let $T_{2,n}=n(\overline{X}(1-\overline{X})-1/4)$ and $Y_2=-Z^2/4$ for $Z\sim \mathrm{N}(0,1)$. We have $\mathbb{E}|V_i|^m=2^{-m}$ for $m\geq1$, and $\mathbb{E}[W^2]=1/4$. Then, applying the bound (\ref{thm2.6bd}) yields, for $h\in C_b^2((-\infty,0))$,
\begin{align*}\Delta_h(T_{2,n},Y_2)\leq 39(\|h'\|+\|h''\|)/n.
\end{align*}
This bound has a faster $O(n^{-1})$ rate of convergence. We can also apply Theorem \ref{thm2.1univ} with $t=2$ to obtain an order $n^{-1/2}$ bound that holds for the wider class of $C_b^1((-\infty,0))$ test functions.		
\end{example}
\begin{example}[Asymptotic distribution of $\overline{X}^p$]\label{ex3.2}
Suppose $X_1,\ldots,X_n$ are i.i.d.\ random variables with mean $\mu$ and variance $\sigma^2\in(0,\infty)$. Define $\overline{X}=n^{-1}\sum_{i=1}^{n}X_i$ and let $p\geq2$. For $i=1,\ldots,n$, let $V_i=X_i-\mu$, so that $V_1,\ldots,V_n$ are i.i.d.\ random variables with zero mean. Denote $\overline{V}=n^{-1}\sum_{i=1}^nV_i$. Define $f(v)=(\mu+v)^p$, so that $f(\overline{V})=\overline{X}^p$.

Suppose first that $\mu\not=0$. We have $f'(0)=p\mu^{p-1}\not=0$, so we can apply Theorem \ref{thm2.1univ} with $t=1$. Moreover, for $v\in\mathbb{R}$,
\begin{align*}|f'(v)|=p|(\mu+v)^{p-1}|\leq 2^{p-2}p(|\mu|^{p-1}+|v|^{p-1})\leq b_{\mu,p}(1+|v|^{p-1}),
\end{align*}
where $b_{\mu,p}:= 2^{p-2}p(1\vee|\mu|^{p-1})$, and we can similarly obtain the bound
$|f''(v)|\leq c_{\mu,p}(1+|v|^{p-2})$, $v\in\mathbb{R}$,
where $c_{\mu,p}:=2^{p-3}p(p-1)(1\vee|\mu|^{p-2})$. We can therefore take $A_1=b_{\mu,p}$, $A_2=c_{\mu,p}$, $r_1=p-1$ and $r_2=p-2$. Let $T_{1,n}=\sqrt{n}(\overline{X}^p-\mu^p)$ and $Y_1\sim \mathrm{N}(0,(p\sigma)^2\mu^{2p-2})$. Then, applying inequality (\ref{thm2.4bd}) we obtain, for $h\in C_1^b(\mathbb{R})$,
\begin{align*}\Delta_h(T_{1,n},Y_1)&\leq\frac{\|h'\|}{\sqrt{n}}\bigg[\frac{c_{\mu,p}}{2}\bigg(\sigma^2+\frac{\mu_{p,\sigma}}{n^{p/2-1}}\bigg)+\frac{3(p+4)b_{\mu,p}}{\sigma^2}\bigg(\bigg(1+\frac{2^{(p-1)/2}\mu_{p,\sigma}}{\sigma}\bigg)\mathbb{E}|V_1|^3\\
&\quad+2^{3(p-1)/2}\mathbb{E}|V_1|^3\mathbb{E}|W|^{p-1}+\mathbb{E}|V_1|^{p+2}\bigg)\bigg].
\end{align*}

Suppose now that $\mu=0$. Then $f(v)=v^p$. We have $f'(0)=\cdots=f^{(p-1)}(0)=0$ and $f^{(p)}(0)=p!\not=0$, so we can apply Theorem \ref{thm2.1univ} with $t=p$. We also have $f^{(p)}(v)=p!=(p!/2)(1+|v|^0)$ and $f^{(p-1)}(v)=0$, so we can take $A_p=p!/2$, $A_{p+1}=r_p=r_{p+1}=0$. Let $T_{p,n}=n^{p/2}\overline{X}^p$ and $Y_p=Z^p$, where $Z\sim \mathrm{N}(0,\sigma^2)$. Then, we obtain
\begin{align*}\Delta_h(T_{p,n},Y_p) &\leq\frac{3p(p+4)\|h'\|}{2\sigma^2\sqrt{n}}\bigg[\bigg(1+\frac{2^{(p-1)/2}\mu_{p,\sigma}}{\sigma}\bigg)\mathbb{E}|X_1|^3\\
&\quad+2^{3(p-1)/2}\mathbb{E}|X_1|^3\mathbb{E}|W|^{p-1}+\mathbb{E}|X_1|^{p+2}\bigg].
\end{align*}
If we additionally assume that $p\geq2$ is an even number then we can obtain a bound with a faster $O(n^{-1})$ convergence rate, provided we make stronger assumptions on the test function $h$. This is because, if $p\geq2$ is an even number, then $f(v)=v^p$ is an even function, and so we can apply Theorem \ref{thm2.2univ} with $t=p$. We can take $A_{p}=p!/2$ and $A_{p+1}=A_{p+2}=r_p=r_{p+1}=r_{p+2}=0$. Applying inequality (\ref{thm2.5bd}) now gives that, for $h\in C_b^2(\mathbb{R}^+)$,
\begin{align*}\Delta_h(T_{p,n},Y_p)&\leq\frac{p^2h_2}{\sigma^2n}\bigg\{\frac{13(2p+3)}{6}\bigg[\bigg(1+\frac{2^{p-1}\mu_{2p-1,\sigma}}{\sigma}\bigg)\mathbb{E}[X_1^4]+2^{3(p-1)}\mathbb{E}[X_1^4]\mathbb{E}|W|^{2p-2}\\
&\quad+\mathbb{E}|X_1|^{2p+2}\bigg]+\frac{3(4p^2-7p+12)}{2\sigma^2}|\mathbb{E}[X_1^3]|\bigg[\bigg(1+\frac{3^{p-1}\mu_{2p-1,\sigma}}{\sigma}\bigg)\mathbb{E}|X_1|^3\\
&\quad+12^{p-1}\mathbb{E}|X_1|^3\mathbb{E}|W|^{2p-2}+\mathbb{E}|X_1|^{2p+1}\bigg]\bigg\}.
\end{align*}
\end{example}

\begin{example}[Asymptotic distribution of the product of two sample moments]\label{ex3.3}
Let $X_{1,j},\ldots,X_{n,j}$, $j=1,2$, be i.i.d.\ random variables with mean $\mathbb{E}[X_{1j}]=\mu_j$ and variance $\mathrm{Var}(X_{1j})=\sigma_j^2\in(0,\infty)$. Define $\overline{X}_j=n^{-1}\sum_{i=1}^nX_{ij}$, $j=1,2$.
For $i=1,\ldots,n$, let $V_{ij}=X_{ij}-\mu_j$, so that $V_{1,j},\ldots,V_{n,j}$, $j=1,2$, are i.i.d.\ random variables with zero mean. Denote $\overline{V}_j=n^{-1}\sum_{i=1}^nV_{ij}$, $j=1,2$, and let $\overline{\mathbf{V}}=(\overline{V}_1,\overline{V}_2)^\intercal$. Define $f(\mathbf{v})=(\mu_1+v_1)(\mu_2+v_2)$, so that $f(\overline{\mathbf{V}})=\overline{X}_1\overline{X}_2$.

Suppose first that $(\mu_1,\mu_2)^\intercal\not=\mathbf{0}$. We have $(\partial_{v_1} f(\mathbf{0}),\partial_{v_2} f(\mathbf{0}))^\intercal=(\mu_2,\mu_1)^\intercal\not=\mathbf{0}$, so we can apply Theorem \ref{thm2.1} with $t=1$. Moreover,
 $\partial_{v_1} f(\mathbf{v})=\mu_2+v_2$, $\partial_{v_2}f(\mathbf{v})=\mu_1+v_1$ and $\partial_{v_1v_2}f(\mathbf{v})=1$, with all other partial derivatives equal to zero. For $j=1,2$, we have $|\partial_{v_j}f(\mathbf{v})|\leq c_{\mu_1,\mu_2}(1+|v_1|+|v_2|)$ for $\mathbf{v}\in\mathbb{R}^2$, where $c_{\mu_1,\mu_2}=\mathrm{max}\{1,|\mu_1|,|\mu_2|\}$. We also have $|\partial_{v_jv_k} f(\mathbf{v})|\leq1= (1/3)(1+|v_1|^0+|v_2|^0)$, $1\leq j,k\leq2$. We may therefore take $A_1=c_{\mu_1,\mu_2}$, $A_2=1/3$, $r_1=1$ and $A_3=r_2=r_3=0$. 
  Let $T_{1,n}=\sqrt{n}(\overline{X}_1\overline{X}_2-\mu_1\mu_2)$ and $Y_1\sim \mathrm{N}(0,(\mu_2\sigma_1)^2+(\mu_1\sigma_2)^2)$. Applying inequality (\ref{thm2.1bd}) now yields, for $h\in C_b^3(\mathbb{R})$,
\begin{align*}\Delta_h(T_{1,n},Y_1)&\leq\frac{1}{\sqrt{n}}\bigg\{\|h'\|(\sigma_1^2+\sigma_2^2)+8c_{\mu_1,\mu_2}^3h_{3}\sum_{j,k=1}^2\bigg[\bigg(1+\frac{8\sigma_{k}^{3}}{\sqrt{\pi}}\bigg)\mathbb{E}|V_{1j}|^3\\
&\quad+\mathbb{E}[V_{1j}^6]+16\sqrt{2}\mathbb{E}|V_{1j}|^3\bigg(3\sigma_{k}^4+\frac{\mathbb{E}[V_{1k}^4]}{n}\bigg)^{3/4}\bigg]\bigg\},
\end{align*}
where, for $W_k=n^{-1/2}\sum_{i=1}^nV_{ik}$, we used H\"older's inequality to bound $\mathbb{E}|W_k|^3\leq(\mathbb{E}[W_k^4])^{3/4}\leq(3\sigma_{k}^4+\mathbb{E}[V_{1k}^4]/n)^{3/4}$ with the final inequality following from a standard calculation that makes use of the fact that $V_{1,k},\ldots,V_{n,k}$ are i.i.d.\ with zero mean and variance $\sigma_{k}^2$.

Now suppose $\mu_1=\mu_2=0$. Then $f(\mathbf{v})=v_1v_2$ is an even function. We also have $(\partial_{v_1} f(\mathbf{0}),\partial_{v_2} f(\mathbf{0}))^\intercal=\mathbf{0}$ and $\partial_{v_1v_2}f(\mathbf{0})=1\not=0$, so we can apply Theorem \ref{thm2.2} with $t=2$. We can take $A_2=1/3$, $A_3=\cdots=A_6=0$ and $r_2=\cdots=r_6=0$ (recall that all partial derivatives of order greater than two are equal to zero). Let $T_{2,n}=n\overline{X}_1\overline{X}_2$ and $Y_2=Z_1Z_2/2$, where $Z_1\sim \mathrm{N}(0,\sigma_1^2)$ and $Z_2\sim \mathrm{N}(0,\sigma_2^2)$ are independent. The limit random variable $Y_2$ is variance-gamma distributed with density $p(x)=(2/(\pi s))K_0(2|x|/s)$, $x\in\mathbb{R}$, where $s=\sigma_1\sigma_2$ and $K_0(x)=\int_0^\infty\mathrm{e}^{-x\cosh(t)}\,\mathrm{d}t$ is a modified Bessel function of the second kind (see \cite[Theorem 1]{g19}). Let $W_k=n^{-1/2}\sum_{i=1}^nX_{ik}$, $k=1,2$.  Applying the bound (\ref{thm2.2bd}) now yields, for $h\in C_b^6(\mathbb{R})$,
\begin{align*}&\Delta_h(T_{2,n},Y_2)\leq\frac{1}{n}\bigg[2630h_{4}\sum_{j,k=1}^2\big((1+120\sigma_k^6)\mathbb{E}[X_{1j}^4]+\mathbb{E}[X_{1j}^{10}]+512\mathbb{E}[X_{1j}^4]\mathbb{E}[W_k^{6}]\big)\\
&\quad+102h_6\sum_{j,a,q=1}^2|\mathbb{E}[X_{1 j}^3]|\big((1+ 810\sigma_q^6)\mathbb{E}|X_{1 a}|^3+\mathbb{E}|X_{1 a}|^{9}+1728\mathbb{E}|X_{1 a}|^3\mathbb{E}[W_q^{6}]\big)\bigg].
\end{align*}
If we additionally assume that $\mathbb{E}[X_{ij}^3]=0$ for all $i,j$,
then we could apply Theorem \ref{thm2.3} to obtain a bound that holds under the weaker assumption that $h\in C_b^4(\mathbb{R})$ and requires only the existence of the eighth order moments of $X_{1,j}$, $j=1,2$.
\end{example}

\begin{example}[Asymptotic joint distribution of sample mean and sample variance]\label{ex3.4}
Let $X_1,\ldots,X_n$ be i.i.d.\ random variables with mean $\mu$ and variance $\sigma^2\in(0,\infty)$. For $i=1,\ldots,n$, let $V_{i,1}=X_i-\mu$, so that $V_{1,1},\ldots,V_{1,n}$ have zero mean. 
Also, for $i=1,\ldots,n$, let $V_{i,2}=V_{i,1}^2-\sigma^2$, so that $V_{1,2},\ldots,V_{n,2}$ have zero mean. Denote $\overline{V}_j=n^{-1}\sum_{i=1}^nV_{ij}$, $j=1,2$, and let $\overline{\mathbf{V}}=(\overline{V}_1,\overline{V}_2)^\intercal$. Let $\mathbf{f}(\mathbf{v})=(f_1(\mathbf{v}),f_2(\mathbf{v}))^\intercal$, where $f_1(\mathbf{v})=\mu+v_1$ and $f_2(\mathbf{v})=\sigma^2+v_2-v_1^2$, so that $\mathbf{f}(\overline{\mathbf{V}})=(\overline{X},n^{-1}\sum_{i=1}^n(X_i-\overline{X})^2)^\intercal$. We have $(\partial_{v_1} f_1(\mathbf{0}),\partial_{v_2} f_1(\mathbf{0}))^\intercal =(1,0)^\intercal$ and $(\partial_{v_1} f_2(\mathbf{0}),\partial_{v_2} f_2(\mathbf{0}))^\intercal =(0,1)^\intercal$, so we can apply Theorem \ref{thm2.1} with $t=1$. Moreover, $\partial_{v_1} f_1(\mathbf{v})=1$, $\partial_{v_2} f_1(\mathbf{v})=0$, $\partial_{v_1} f_2(\mathbf{v})=-2v_1$, $\partial_{v_2} f_2(\mathbf{v})=1$, and $\partial_{v_1}^2 f_2(\mathbf{v})=-2$ with all other partial derivatives equal to zero. By our usual arguments, we may therefore take $A_1=2$, $A_2=2/3$, $r_1=1$ and $A_3=r_2=r_3=0$. Let $\mathbf{T}_{1,n}=\sqrt{n}(\overline{X}-\mu,n^{-1}\sum_{i=1}^n(X_i-\overline{X})^2-\sigma^2)^\intercal$ and $\mathbf{Y}_1\sim \mathrm{N}_2(\mathbf{0},\Sigma)$, where $(\Sigma)_{ij}=\sigma_{ij}$ with $\sigma_{11}=\sigma^2$, $\sigma_{12}=\sigma_{21}=\mathbb{E}[(X_1-\mu)^3]=\mu_3$ and $\sigma_{22}=\mathrm{Var}((X_1-\mu)^2)=\mu_4-\sigma^4$, and $\mu_3$ and $\mu_4$ denote the third and fourth central moments of $X_1$, respectively. Then, applying inequality (\ref{thm2.1bd}) gives that, for $h\in C_b^3(\mathbb{R}^2)$,
\begin{align*}\Delta_h(\mathbf{T}_{1,n},\mathbf{Y}_1)&\leq\frac{4}{\sqrt{n}}\bigg\{|h|_1(\mu_4+\sigma^2-\sigma^4)+16h_{3,2}\sum_{j,k=1}^2\bigg[\bigg(1+\frac{8\sigma_{kk}^{3/2}}{\sqrt{\pi}}\bigg)\mathbb{E}|V_{1j}|^3\\
&\quad+\mathbb{E}[V_{1j}^6]+16\sqrt{2}\mathbb{E}|V_{1j}|^3\bigg(3\sigma_{kk}^2+\frac{\mathbb{E}[V_{1k}^4]}{n}\bigg)^{3/4}\bigg]\bigg\},
\end{align*}
where we simplified the bound similarly to how we did in Example \ref{ex3.3}.

We remark that the approach used in this example could be applied to obtain $O(n^{-1/2})$ Wasserstein distance bounds for the normal approximation of the moment estimators $n^{-1}\sum_{i=1}^n(X_i-\overline{X})^p$, $p\geq2$, as well as $O(n^{-1/2})$ bounds for the multivariate normal approximation of the joint distribution of various moment estimators. 
\end{example}

\begin{example}[Rank-based statistics]\label{ex3.5}
 Suppose we have $r\geq2$ treatments across $n$ independent trials and for the $i$-th trial we have the ranking $\pi_i(1),\ldots,\pi_i(r)$, where $\pi_i(j)\in\{1,\ldots,r\}$, over the $r$ treatments.  Under the null hypothesis, the rankings are independent permutations $\pi_1,\ldots,\pi_n$, with each permutation being equally likely. Fix a function $J:\{1,\ldots,r\}\rightarrow\mathbb{R}$ and let 
\begin{equation*}\overline{J}=\frac{1}{r}\sum_{k=1}^rJ(k), \quad \sigma_J^2=\frac{1}{r-1}\sum_{k=1}^r(J(k)-\overline{J})^2.
\end{equation*} 
For $i=1,\ldots,n$ and $j=1,\ldots,r$, let $X_{ij}^J=(J(\pi_i(j))-\overline{J})/\sigma_J$ and  $W_j^J=n^{-1/2}\sum_{i=1}^nX_{ij}^J$. Then, Sen \cite{sen} proposed the test statistic
 \begin{equation*}S_{n,r}^J=\sum_{j=1}^r(W_j^J)^2,
 \end{equation*}
and proved that it is asymptotically $\chi_{(r-1)}^2$ distributed under the null hypothesis.

Friedman's statistic \cite{friedman}, which we denote by $F_{n,r}$, corresponds to the case $J(k)=k$, and in this case $\overline{J}=(r+1)/2$ and $\sigma_J^2=r(r+1)/12$. Another special case is Brown-Mood's statistic, introduced by \cite{bm51} as an alternative to Friedman's statistic. This statistic
corresponds to the case $J(k)=\mathbf{1}(k\leq a)$ for a fixed $a\in\{1,\ldots,r-1\}$, and we have $\overline{J}=a/r$ and $\sigma_J^2=a(r-a)/(r(r-1))$.

Kolmogorov distance bounds to quantify the chi-square approximation for Sen's family of rank-based statistics were derived by \cite{pu21} under the assumption that for all $k=1,\ldots,r$ we have $|J(k)-\overline{J}|\leq B$, for some constant $B$ (possibly depending on $r$). Note that both the Friedman and Brown-Mood statistics satisfy this assumption. A Kolmogorov distance bound of order $n^{-1}$ (without explicit dependence on $r$) was derived by \cite{pu21} for the case $r\geq6$, as well as a bound on the rate of convergence of order $n^{-1+1/r}$ for $2\leq r\leq 5$. In this example, we obtain $O(n^{-1})$ bounds for smooth test functions that hold for all $r\geq2$.

Define $\overline{X}_j=n^{-1}\sum_{i=1}^nX_{ij}$, where we have dropped the superscript $J$ in the notation. For fixed $j=1,\ldots,r$, the random variables $\pi_1(j),\ldots,\pi_n(j)$ are i.i.d$.$ with uniform distribution on $\{1,\ldots,r\}$. Thus, for fixed $j$, the random variables $X_{1j},\ldots,X_{nj}$ are also i.i.d..
Observe that we have the representation $S_{n,r}^J=nf(\overline{\mathbf{X}})$, where $f(\mathbf{x})=\sum_{j=1}^rx_j^2$ (with $f(\mathbf{0})=0$). For $j=1,\ldots,r$, we also have $\partial_{x_j} f(\mathbf{x})=2x_j$ and $\partial_{x_j}^2 f(\mathbf{x})=2$ with all other partial derivatives equal to zero. In particular, $\partial_{x_j} f(\mathbf{0})=0$ and $\partial_{x_j}^2 f(\mathbf{0})=2$, for $j=1,\ldots,r$. Also, the covariance matrix of $\mathbf{W}=(W_1,\ldots,W_r)^\intercal$ is non-negative definite with entries $\sigma_{jj}=(r-1)/r$ and $\sigma_{jk}=-1/r$, $j\not=k$ (see the SM).  Moreover, $f$ is an even function. We can therefore apply Theorem \ref{thm2.2} with $t=2$. It is readily seen that we can take $A_2=2$, $r_2=1/6,$ $A_3=\cdots A_6=r_3=\cdots=r_6=0$ (we take $r_2=1/6$ to obtain a bound with an improved dependence on the parameter $r$; see Remark \ref{vfnd}). We immediately have the bound $\mathbb{E}|X_{i,j}|^k\leq (B\sigma_J)^k$, for all $i,j$ and $k\geq1$, and an application of the Marcinkiewiczi-Zygmund inequality \cite{mz37} gives that $\mathbb{E}|W_j|^k\leq K(B/\sigma_J)^k$ for some constant $K$. 
Let $V_r\sim\chi_{(r-1)}^2$. Now, applying inequality (\ref{thm2.2bd}) gives, for $h\in C_b^6(\mathbb{R}^+)$,  
\begin{align}\label{fr1}\Delta_h(S_{n,r}^J,V_r)\leq Cr^{12}(B/\sigma_J)^{13}h_6/n,
\end{align}
for some universal constant $C>0$, where to obtain a compact final bound we used that $1\leq B/\sigma_J$.

Since all third order partial derivatives of $f$ are equal to zero, we can also apply part (i) of Theorem \ref{thmsec2} to obtain an alternative bound. 
We can write $S_{n,r}^J=g(\mathbf{W})$, where $g(\mathbf{w})=\sum_{j=1}^rw_j^2$.  
For $k=1,\ldots,r$, we have $\partial_{w_k} g(\mathbf{w})=2w_k$, $\partial_{w_k}^2g(\mathbf{w})=2$ and all other derivatives are equal to 0.  Now, $|\partial_{w_k} g(\mathbf{w})|^6=64w_k^6$ and $|\partial_{w_k}^2 g(\mathbf{w})|^3=8$, meaning we can take $P(\mathbf{w})=8+64\sum_{j=1}^rw_j^6$ as our dominating function.  Applying the bound (\ref{boundthm2}) with $A=8$, $B=64$ and $r=6$ yields, for $h\in C_b^6(\mathbb{R}^+)$, 
\begin{align}\label{fr20}\Delta_h(S_{n,r}^J,V_r)\leq C'r^7(B/\sigma_J)^{12}h_6/n,
\end{align}
for some universal constant $C'>0$. The bound (\ref{fr20}) has a better dependence on the number of treatments $r$ than the bound (\ref{fr1}).

For Friedman's statistic, the distribution of the random variables $X_{ij}$ is symmetric about 0, so $\mathbb{E}[X_{ij}X_{ik}X_{il}]=0$ for all $1\leq i\leq n$ and $1\leq j,k,l\leq r$. We can therefore apply part (i) of Theorem \ref{thmsec2} to obtain an improved bound. 
We have $|\partial_{w_k} g(\mathbf{w})|^4=16w_k^4$ and $|\partial_{w_k}^2 g(\mathbf{w})|^2=4$, so we can take $P(\mathbf{w})=4+16\sum_{j=1}^rw_j^4$ as our dominating function.  Applying inequality (\ref{boundtheorem}) with $A=4$, $B=16$ and $r=4$ yields, for $h\in C_b^4(\mathbb{R}^+)$, 
\begin{align*}\Delta_h(F_{n,r},V_r)\leq C''r^5h_4/n,
\end{align*}
for some universal constant $C''>0$. This improves on the bound $Cr^7h_6/n$ that would result from specialising the bound (\ref{fr20}) to the case of Friedman's statistic. However, the dependence on $r$ is worse than the bound of \cite[Theorem 1.1]{gr21} which also has the optimal $O(n^{-1})$ rate for smooth test functions, but tends to zero under the optimal condition $n/r\rightarrow\infty$. It is to be expected that the approach of \cite{gr21} that was specifically targeted to the chi-square approximation of Friedman's statistic results in a better bound, although it is interesting that we are able to derive a bound with the optimal $O(n^{-1})$ rate so efficiently as a consequence of our general bounds.
\end{example}

\begin{example}[Pearson's statistic]\label{ex3.6} Consider $n$ independent trials, with each trial resulting in a unique classification over $r$ classes. Let $p_1,\ldots, p_r$ be the non-zero classification probabilities, and let $(U_1,\ldots , U_r)$ denote the observed numbers in each class. Then Pearson's chi-square statistic for goodness-of-fit \cite{pearson}, given by
\begin{equation*} \chi^2 = \sum_{j=1}^r \frac{(U_j - n p_j)^2}{n p_j}, 
\end{equation*}
is asymptotically $\chi_{(r-1)}^2$ distributed under the null hypothesis.  

The cell counts $U_j\sim\mathrm{Bin}(n,p_j)$, $1\leq j\leq r$, are dependent random variables that satisfy $\sum_{j=1}^rU_j=n$.  Let $I_{ij}\sim \mathrm{Ber}(p_j)$ denote the indicator that the $i$-th trial falls in the $j$-th cell.  By letting $X_{ij}=(I_{ij}-p_j)/\sqrt{p_j}$ and $W_j=n^{-1/2}\sum_{i=1}^nX_{ij}$, we can write $\chi^2=g(\mathbf{W})=\sum_{j=1}^rW_j^2$,
where $g(\mathbf{w})=\sum_{j=1}^rw_j^2$.  As was the case for the rank-based statistics, the function $g$ is even and possesses derivatives of polynomial growth.  Also, since trials are assumed to be independent, the random variables $X_{1,j},\ldots,X_{n,j}$ are independent for fixed $j$.  The covariance matrix of $\mathbf{W}=(W_1,\ldots,W_r)^\intercal$ is non-negative definite, with entries $\sigma_{jj}=1-p_j$ and $\sigma_{jk}=-\sqrt{p_jp_k}$, $j\not=k$ (see \cite{gaunt chi square}).  Therefore Pearson's statistic falls within the class of statistics covered by part (i) of Theorem \ref{thmsec2}.  

As was the case for the rank-based statistics, $|\partial_{w_k} g(\mathbf{w})|^6=64w_k^6$ and $|\partial_{w_k}^2 g(\mathbf{w})|^3=8$, so we can take $P(\mathbf{w})=8+64\sum_{j=1}^rw_j^6$ as our dominating function.  
Applying inequality (\ref{boundthm2}) with $A=8$, $B=64$ and $r=6$ results in a bound of the form
\[\Delta_h(\chi^2,V_r)\leq Cn^{-1}h_6,
\]
where $C>0$ is a constant depending on $r$ and $p_1,\ldots,p_r$, but not $n$.  We do not explicitly find such a $C$ since a superior upper bound of the form $Kr(np_*)^{-1}\sum_{k=0}^5\|h^{(k)}\|$, where $p_*=\min_{1\leq j\leq r}p_j$ and $h^{(0)}\equiv h$, was obtained by \cite{gaunt chi square}.  This bound has a much better dependence on $r$ and $p_1,\ldots,p_r$ than a bound that would result from an application of Theorem \ref{thmsec2}. As was the case with Friedman's statistic, it is to be expected that the bound of \cite{gaunt chi square} outperforms ours, given that their approach was targeted specifically to Pearson's statistic rather than the more general class of statistics considered in this paper. 
\end{example}		
		
\section{Stein's method for functions of multivariate normal random vectors}\label{sec4}


We begin this section by outlining how Stein's method can be used to prove approximation theorems when the limit distribution can be represented as a function of multivariate normal random vectors (for more details see \cite{gaunt normal}). 
Let $\mathbf{g}:\mathbb{R}^d\rightarrow\mathbb{R}^m$ be continuous and let $\mathbf{Z}\sim \mathrm{N}_d(\mathbf{0},\Sigma)$, where the covariance matrix $\Sigma$ is non-negative definite.  Suppose that we are interested in bounding the distance between the distributions of $\mathbf{g}(\mathbf{W})$ and $\mathbf{g}(\mathbf{Z})$. 
Consider the multivariate normal Stein equation \cite{barbour2,gotze}
  with test function $h(\mathbf{g}(\cdot))$:
\begin{equation} \label{mvng} \nabla^\intercal\Sigma\nabla f(\mathbf{w})-\mathbf{w}^\intercal\nabla f(\mathbf{w})=h(\mathbf{g}(\mathbf{w}))-\mathbb{E}[h(\mathbf{g}(\mathbf{Z}))].
\end{equation} 
We can therefore bound the quantity of interest $|\mathbb{E}[h(\mathbf{g}(\mathbf{W}))]-\mathbb{E}[h(\mathbf{g}(\mathbf{Z}))]|$ by solving (\ref{mvng}) for $f$ and then bounding the expectation 
\begin{equation}\label{emvn}\mathbb{E}[\nabla^\intercal\Sigma\nabla f(\mathbf{W})-\mathbf{W}^\intercal\nabla f(\mathbf{W})].
\end{equation}
A number of coupling techniques have been developed for bounding such expectations (see \cite{chatterjee 3, goldstein 2, goldstein1, meckes, reinert 1}).  These papers also give general plug-in bounds for this quantity, although these only hold for the classical case that the derivatives of the test function (here $h(\mathbf{g}(\cdot))$) are bounded, in which standard bounds for the derivatives of the solution to (\ref{mvng}) can be applied (see \cite{gaunt rate, goldstein1, meckes}).   However, in general, the derivatives of the test function $h(\mathbf{g}(\cdot))$ will be unbounded (for example, consider $g(w)=w^2$) and therefore the derivatives of the solution (see \cite{barbour2,goldstein1})
\begin{equation}\label{mvnsolnh}f(\mathbf{w})=-\int_{0}^{\infty}\{\mathbb{E}[h(\mathbf{g}(\mathrm{e}^{-s}\mathbf{w}+\sqrt{1-\mathrm{e}^{-2s}}\mathbf{Z}))]-\mathbb{E}[h(\mathbf{g}(\mathbf{Z}))]\}\,\mathrm{d}s
\end{equation}
will also in general be unbounded.  The partial derivatives of the solution (\ref{mvnsolnh}) were bounded by \cite{gaunt normal,GauntSut} for a large class of functions $g:\mathbb{R}^d\rightarrow\mathbb{R}$; in particular, bounds are given for the case that the partial derivatives of $g$ have polynomial growth. 
With such bounds on the solution and the coupling strategies developed for multivariate normal approximation it is in principle possible to bound the expectation (\ref{emvn}).
 This is the approach we shall take 
 in the remainder of this section.

\subsection{Bounds for the solution of the Stein equation for $\mathbf{g}:\mathbb{R}^d\rightarrow\mathbb{R}^m$}
		
The bounds of \cite{gaunt normal,GauntSut} for the solution of the Stein equation (\ref{mvnsolnh}) were given for real-valued functions $g:\mathbb{R}^d\rightarrow\mathbb{R}$. In this section, we generalise these bounds to vector-valued functions $\mathbf{g}:\mathbb{R}^d\rightarrow\mathbb{R}^m$.	
The key to this generalisation is Lemma \ref{bell lem} below, in which we obtain a simple bound for the $n$-th order partial derivatives of $\phi(\mathbf{w})=h(\mathbf{g}(\mathbf{w}))$.  To achieve this bound, we make use of a generalisation of the Fa\`{a} di Bruno formula to arbitrary $d,m\geq1$, due to \cite{cs96}.  A formula of \cite{ma} which gives a generalisation of the classical Fa\`{a} di Bruno formula to arbitrary $d\geq1$ and $m=1$ was used in \cite{gaunt normal}.  
		
		Before stating the formula of \cite{cs96}, we recall some multivariate notation that will help to simplify expressions.  Let $\mathbb{N}$ be the set of natural numbers.  For $\boldsymbol{\nu}=(\nu_1,\ldots,\nu_d)\in\mathbb{N}^d$ and $\mathbf{z}\in (z_1,\ldots,z_d)\in\mathbb{R}^d$, we write
		\begin{align*}|\boldsymbol{\nu}|&=\sum_{i=1}^d\nu_i, \quad \boldsymbol{\nu}!=\prod_{i=1}^d(\nu_i!), \quad D_{\mathbf{x}}^{\boldsymbol{\nu}}=\frac{\partial^{|\boldsymbol{\nu}|}}{\partial x_1^{\nu_1}\cdots\partial x_d^{\nu_d}}, \quad \text{for $|\boldsymbol{\nu}|>0$}, \\
			D_\mathbf{x}^\mathbf{0}&=\text{identity operator}, \quad \mathbf{z}^{\boldsymbol{\nu}}=\prod_{i=1}^d z_i^{\nu_i}.
		\end{align*}
		If $(l_1,\ldots,l_d)\in\mathbb{N}^d$, we write $\mathbf{l}\leq\boldsymbol{\nu}$ provided $l_i\leq \nu_i$ for $i=1,\ldots,d$.  This allows us to introduce a linear ordering on $\mathbb{N}^d$.  If $\boldsymbol{\mu}=(\mu_1,\ldots,\mu_d)$ and $\boldsymbol{\nu}=(\nu_1,\ldots,\nu_d)$ are in $\mathbb{N}^d$, we write $\boldsymbol{\mu}\prec \boldsymbol{\nu}$ provided one of the following holds: (i) $|\boldsymbol{\mu}|<|\boldsymbol{\nu}|$; (ii) $|\boldsymbol{\mu}|=|\boldsymbol{\nu}|$ and $\mu_1<\nu_1$; or (iii) $|\boldsymbol{\mu}|=|\boldsymbol{\nu}|$, $\mu_1=\nu_1,\ldots,\mu_k=\nu_k$ and $\mu_{k+1}<\nu_{k+1}$ for some $1\leq k\leq d-1$.		
		
		Finally, we set $\phi_{\boldsymbol{\nu}}=D_{\mathbf{x}}^{\boldsymbol{\nu}}\phi(\mathbf{x})$, $h_{\boldsymbol{\lambda}}=D_\mathbf{y}^{\boldsymbol{\lambda}}(\mathbf{y})$, $g_{\boldsymbol{\mu}}^{(i)}=D_\mathbf{x}^{\boldsymbol{\mu}}g^{(i)}(\mathbf{x})$ and $\mathbf{g}_{\boldsymbol{\mu}}=(g_{\boldsymbol{\mu}}^{(1)},\ldots,g_{\boldsymbol{\mu}}^{(m)})$.  With this notation, we can state the generalised Fa\`{a} di Bruno formula (\cite[Theorem 2.1]{cs96}):
		\begin{equation}\label{bdf}\phi_{\boldsymbol{\nu}}=\sum_{1\leq|\boldsymbol{\lambda}|\leq n}h_{\boldsymbol{\lambda}} \sum_{s=1}^n\sum_{p_s(\boldsymbol{\nu},\boldsymbol{\lambda})}(\boldsymbol{\nu}!)\prod_{j=1}^s\frac{[\mathbf{g}_{\mathbf{l}_j}]^{\mathbf{k}_j}}{(\mathbf{k}_j!)[\mathbf{l}_j!]^{|\mathbf{k}_j|}},
		\end{equation}
		where $n=|\boldsymbol{\nu}|$ and
		\begin{align*}&p_s(\boldsymbol{\nu},\boldsymbol{\lambda})=\{(\mathbf{k}_1,\ldots,\mathbf{k}_s;\mathbf{l}_1,\ldots,\mathbf{l}_s)\,:\,|\mathbf{k}_i|>0, \\
			&\mathbf{0}\prec\mathbf{l}_1\prec\cdots\prec\mathbf{l}_s, \: \sum_{i=1}^s\mathbf{k}_i=\boldsymbol{\lambda} \: \text{and} \: \sum_{i=1}^s|\mathbf{k}_i|\mathbf{l}_i=\boldsymbol{\nu}\}.
		\end{align*}
		In the above, the vectors $\mathbf{k}$ are $m$-dimensional, the vectors $\mathbf{l}$ are $d$-dimensional
		and we always set $0^0 = 1$.  We shall also make use of the following formula \cite[Corollary 2.9]{cs96}:
		\begin{equation}\label{ineq1}\sum_{|\boldsymbol{\lambda}|=k} \sum_{s=1}^n\sum_{p_s(\boldsymbol{\nu},\boldsymbol{\lambda})}(\boldsymbol{\nu}!)\prod_{j=1}^s\frac{1}{(\mathbf{k}_j!)[\mathbf{l}_j!]^{|\mathbf{k}_j|}}=m^k{n\brace k},
		\end{equation}
		where we recall that ${n\brace k}$ is a Stirling number of the second kind.
		
		With formulas (\ref{bdf}) and (\ref{ineq1}) we are able to obtain a simple bound for the $n$-th order partial derivatives. 
		For a given function $P:\mathbb{R}^d\rightarrow\mathbb{R}^+$, we recall that the function $\mathbf{g}:\mathbb{R}^d\rightarrow\mathbb{R}^m$ belongs to the class $C_{P}^{n,m}(\mathbb{R}^d)$ if all $n$-th order partial derivatives of $g^{(1)},\ldots,g^{(m)}$ exist and are such that, for $\mathbf{w}\in\mathbb{R}^d$, 
		\[|g_{\mathbf{l}_j}^{(i)}(\mathbf{w})|^{n/|\mathbf{l}_j|}\leq P(\mathbf{w}), \quad |\mathbf{l}_j|=1,\ldots,n,\: i=1,\ldots,m.
		\]
		
		\begin{lemma}\label{bell lem}Suppose $h\in C_b^{n}(\mathbb{R}^m)$ and $\mathbf{g}\in C_P^{n,m}(\mathbb{R}^d)$.  Then, for $\boldsymbol{\nu}\in\mathbb{N}^d$ such that $|\boldsymbol{\nu}|=n$, and $\mathbf{w}\in\mathbb{R}^d$, 
			\begin{equation}\label{hgdiff0}|\phi_{\boldsymbol{\nu}}(\mathbf{w})|\leq h_{n,m}P(\mathbf{w}).
			\end{equation}
		\end{lemma}
		
\begin{remark}\label{remm1}Lemma \ref{bell lem} generalises Lemma 2.1 of \cite{gaunt normal} from $m=1$ to $m\geq1$ in a very simple manner.  The bound (\ref{hgdiff0})  for $m\geq1$ takes  the same form as inequality (2.1) of \cite{gaunt normal}, with the only difference being that the factor $h_{n,m}$ is present instead of the factor $h_n:=h_{n,1}$.  Inequality (2.1) of Lemma 2.1 of \cite{gaunt normal} was the starting point to bound the partial derivatives of the solution (\ref{mvnsolnh}) (in the case $m=1$).
 Since our generalised bound only differs by a multiplicative factor independent of $\mathbf{w}$, one can extend all the bounds given in Section 2 of \cite{gaunt normal} and Section 2 of \cite{GauntSut} to the case $m\geq1$ by simply replacing the factor $h_i=h_{i,1}$ by $h_{i,m}$.
\end{remark}
		
		
\begin{proof} Since $\mathbf{g}\in C_P^{n,m}(\mathbb{R}^d)$, we have that
			\begin{equation*}\prod_{j=1}^s|\mathbf{g}_{\mathbf{l}_j}(\mathbf{w})|^{\mathbf{k}_j}=\prod_{j=1}^s\prod_{i=1}^m|\mathbf{g}_{\mathbf{l}_j}^{(i)}(\mathbf{w})|^{(\mathbf{k}_j)_i}\leq\prod_{j=1}^s\prod_{i=1}^m\big((P(\mathbf{w}))^{|\mathbf{l}_j|/n}\big)^{(\mathbf{k}_j)_i},
			\end{equation*}
			where $(\mathbf{k}_j)_i$ is the $i$-th component of $\mathbf{k}_j$.  But
			\begin{equation*}\sum_{j=1}^s\sum_{i=1}^m|\mathbf{l}_j|(\mathbf{k}_j)_i=\sum_{j=1}^s|\mathbf{l}_j|\sum_{i=1}^m(\mathbf{k}_j)_i=\sum_{j=1}^s|\mathbf{l}_j||\mathbf{k}_j|=n,
			\end{equation*}
			and therefore, for $\mathbf{w}\in\mathbb{R}^d$,
			\begin{equation}\label{ineq2}\prod_{j=1}^s|\mathbf{g}_{\mathbf{l}_j}(\mathbf{w})|^{\mathbf{k}_j}\leq P(\mathbf{w}).
			\end{equation}
			Recalling formula (\ref{bdf}), we have that
			\begin{align*}|\phi_{\boldsymbol{\nu}}(\mathbf{w})|&=\bigg|\sum_{1\leq|\boldsymbol{\lambda}|\leq n}h_{\boldsymbol{\lambda}}(\mathbf{g}(\mathbf{w})) \sum_{s=1}^n\sum_{p_s(\boldsymbol{\nu},\boldsymbol{\lambda})}(\boldsymbol{\nu}!)\prod_{j=1}^s\frac{[\mathbf{g}_{\mathbf{l}_j}(\mathbf{w})]^{\mathbf{k}_j}}{(\mathbf{k}_j!)[\mathbf{l}_j!]^{|\mathbf{k}_j|}}\bigg| \\
				&\leq\sum_{1\leq|\boldsymbol{\lambda}|\leq n}|h|_{|\boldsymbol{\lambda}|} \sum_{s=1}^n\sum_{p_s(\boldsymbol{\nu},\boldsymbol{\lambda})}(\boldsymbol{\nu}!)\prod_{j=1}^s\frac{|\mathbf{g}_{\mathbf{l}_j}(\mathbf{w})|^{\mathbf{k}_j}}{(\mathbf{k}_j!)[\mathbf{l}_j!]^{|\mathbf{k}_j|}}\\
				&\leq P(\mathbf{w})\sum_{1\leq|\boldsymbol{\lambda}|\leq n}|h|_{|\boldsymbol{\lambda}|} \sum_{s=1}^n\sum_{p_s(\boldsymbol{\nu},\boldsymbol{\lambda})}(\boldsymbol{\nu}!)\prod_{j=1}^s\frac{1}{(\mathbf{k}_j!)[\mathbf{l}_j!]^{|\mathbf{k}_j|}} \\
				&=P(\mathbf{w})\sum_{1\leq |\boldsymbol{\lambda}|\leq n}m^{|\boldsymbol{\lambda}|}{n\brace |\boldsymbol{\lambda}|}|h|_{|\boldsymbol{\lambda}|},
			\end{align*}
			where we used inequality (\ref{ineq2}) to obtain the final inequality and we used (\ref{ineq1}) to obtain the final equality.  This completes the proof.
\end{proof}

\subsection{Proofs of Theorems \ref{theoremsec21}--\ref{indep2}}


The following lemma provides bounds on some expectations that will appear in the proofs of Theorems \ref{theoremsec21} and \ref{thmsec2}. The proof is given in the SM. In the lemma, $f$ denotes the solution (\ref{mvnsolnh}) to the Stein equation (\ref{mvng}) and $\psi_{jkl}$ is the solution to the following partial differential equation that was introduced by \cite{gaunt normal}:
\begin{equation}\label{234multinor}\nabla^\intercal\Sigma\nabla \psi_{jkl}(\mathbf{w})-\mathbf{w}^\intercal\nabla \psi_{jkl}(\mathbf{w})=\frac{\partial^{3} f}{\partial w_j\partial w_k\partial w_l}(\mathbf{w}).
\end{equation}
For $i=1,\ldots,n$ and $j=1,\ldots,d$, we introduce the random variables $W_j^{(i)}=W_j-n^{-1/2}X_{ij}$, so that $W_j^{(i)}$ and $X_{ij}$ are independent.  We also write $\mathbf{W}^{(i)}=(W_1^{(i)},\ldots,W_d^{(i)})^\intercal$. 

\begin{lemma}\label{cbwbshc} Let $\mathbf{X}_i$ denote the random vector $(X_{i,1},\ldots,X_{i,d})^\intercal$ and, for $r\geq0$, let $u:\mathbb{R}^d\rightarrow\mathbb{R}^+$ be such that $\mathbb{E}|X_{ij}^{r}u(\mathbf{X}_i)|<\infty$ for all $i=1,\ldots,n$ and $j=1,\ldots,d$. Assume that $\Sigma$ is non-negative definite.  Then, for $\theta\in(0,1)$,
\begin{align}\mathbb{E}\bigg|u(\mathbf{X}_i)\frac{\partial^tf}{\prod_{j=1}^{t}\partial w_{i_j}}(\mathbf{W}_{\theta}^{(i)})\bigg| &\leq \frac{h_{t,m}}{t}\Big[A\mathbb{E}[u(\mathbf{X}_i)]+B\sum_{j=1}^d\big(2^{3r/2}\mathbb{E}[u(\mathbf{X}_i)]\mathbb{E}|W_j|^{r} \nonumber \\
\label{cbhsxx}&\quad+\mathbb{E}|X_{ij}^{r}u(\mathbf{X}_i)|+2^{r/2}\mu_{r,\sigma_j}\mathbb{E}[u(\mathbf{X}_i)]\big)\Big], \\
\mathbb{E}\bigg|u(\mathbf{X}_i)\frac{\partial^3\psi_{abc}}{\partial w_j\partial w_k\partial w_l}(\mathbf{W}_{\theta}^{(i)})\bigg| &\leq \frac{h_{6,m}}{18}\Big[A\mathbb{E}[u(\mathbf{X}_i)]+B\sum_{j=1}^d\big(12^{r/2}\mathbb{E}[u(\mathbf{X}_i)]\mathbb{E}|W_j|^{r} \nonumber \\
\label{cbhsxx6}&\quad+\mathbb{E}|X_{ij}^{r}u(\mathbf{X}_i)|+2\cdot3^{r/2}\mu_{r,\sigma_{j}}\mathbb{E}[u(\mathbf{X}_i)]\big)\Big],
\end{align}
where the inequalities are for $\mathbf{g}$ in the classes $C_P^{t,m}(\mathbb{R}^d)$ and $C_P^{6,m}(\mathbb{R}^d)$, respectively, and $\mathbf{W}_\theta^{(i)}=\mathbf{W}^{(i)}+\theta n^{-1/2}\mathbf{X}_i. $ Inequality (\ref{cbhsxx}) is valid for $n\geq8$, whilst inequality (\ref{cbhsxx6}) holds for $n\geq12$.
\end{lemma}

 
\noindent{\emph{Proof of Theorem \ref{theoremsec21}.}} 
Taylor expanding $\frac{\partial f}{\partial w_j}(\mathbf{W})$ about $\mathbf{W}^{(i)}$, using independence and that $\mathbb{E}[X_{ij}]=0$ for all $i,j$ gives
\begin{align*}&\sum_{j=1}^d\mathbb{E}\bigg[W_j\frac{\partial f}{\partial w_j}(\mathbf{W})\bigg]=\frac{1}{\sqrt{n}}\sum_{i=1}^n\sum_{j=1}^d\mathbb{E}\bigg[X_{ij}\frac{\partial f}{\partial w_j}(\mathbf{W})\bigg] \\
&=\frac{1}{\sqrt{n}}\sum_{i=1}^n\sum_{j=1}^d\mathbb{E}[X_{ij}]\mathbb{E}\bigg[\frac{\partial f}{\partial w_j}(\mathbf{W}^{(i)})\bigg]+\frac{1}{n}\sum_{i=1}^n\sum_{j,k=1}^d\mathbb{E}[X_{ij}X_{ik}]\mathbb{E}\bigg[\frac{\partial^2f}{\partial w_j\partial w_k}(\mathbf{W}^{(i)})\bigg]+R_1 \\
&
=\sum_{j,k=1}^d\sigma_{jk}\mathbb{E}\bigg[\frac{\partial^2f}{\partial w_j\partial w_k}(\mathbf{W})\bigg]+R_1+R_2,
\end{align*}
where
\begin{align*}|R_1|&\leq\frac{1}{2n^{3/2}}\sum_{i=1}^n\sum_{j,k,l=1}^d\mathbb{E}\bigg|X_{ij}X_{ik}X_{il}\frac{\partial^3f}{\partial w_j\partial w_k\partial w_l}(\mathbf{W}_{\theta_1}^{(i)})\bigg|, \\
|R_2|&\leq\frac{1}{n^{3/2}}\sum_{i=1}^n\sum_{j,k,l=1}^d|\mathbb{E}[X_{ij}X_{ik}]|\mathbb{E}\bigg|X_{il}\frac{\partial^3f}{\partial w_j\partial w_k\partial w_l}(\mathbf{W}_{\theta_2}^{(i)})\bigg|,
\end{align*}
and $\mathbf{W}_{\theta}^{(i)}=\mathbf{W}^{(i)}+\theta n^{-1/2}\mathbf{X}_{i}$ for some $\theta\in(0,1)$. Here we used that $n^{-1}\sum_{i=1}^n\mathbb{E}[X_{ij}X_{ik}]=\mathbb{E}[W_jW_k]=\sigma_{jk}$.  Bounding the expectations in $|R_1|$ and $|R_2|$ using Lemma \ref{cbwbshc} gives that
\begin{align}&|\mathbb{E}[h(\mathbf{g}(\mathbf{W}))]-\mathbb{E}[h(\mathbf{g}(\mathbf{Z}))]|\leq |\mathbb{E}[\nabla^\intercal\Sigma\nabla f(\mathbf{W})-\mathbf{W}^\intercal\nabla f(\mathbf{W})]|\leq |R_1|+|R_2|\nonumber\\
&\leq\frac{h_{3,m}}{6n^{3/2}}\sum_{i=1}^n\sum_{j,k,l=1}^d\bigg\{A\mathbb{E}|X_{ij}X_{ik}X_{il}|+B\sum_{t=1}^d \big(2^{3r/2}\mathbb{E}|X_{ij}X_{ik}X_{il}|\mathbb{E}|W_t|^{r}\nonumber\\
&\quad+\mathbb{E}|X_{ij}X_{ik}X_{il}X_{it}^{r}|+2^{r/2}\mu_{r,\sigma_t}\mathbb{E}|X_{ij}X_{ik}X_{il}|\big)+2|\mathbb{E}[X_{ij}X_{ik}]|\bigg[A\mathbb{E}|X_{il}| \nonumber\\
&\quad+B\sum_{t=1}^d \big(2^{3r/2}\mathbb{E}|X_{il}|\mathbb{E}|W_t|^{r}+\mathbb{E}|X_{il}X_{it}^{r}|+2^{r/2}\mu_{r,\sigma_t}\mathbb{E}|X_{il}|\big)\bigg]\bigg\}.\nonumber
\end{align}
The simplified bound (\ref{genbd1}) is obtained by applying the following considerations.  
We use Young's inequality for products ($ab\leq a^p/p+b^q/q$, $a,b\geq0$, $p,q>1$ such that $p^{-1}+q^{-1}=1$) to bound expectations such as $\mathbb{E}|X_{ij}X_{ik}X_{il}X_{it}^{r}|\leq c_1\mathbb{E}|X_{ij}|^{r+3}+c_2\mathbb{E}|X_{ik}|^{r+3}+c_3\mathbb{E}|X_{il}|^{r+3}+c_4\mathbb{E}|X_{it}|^{r+3}$, where $c_1+c_2+c_3+c_4=1$. We also use Young's inequality for products to bound quantities such as $|\mathbb{E}[X_{ij}X_{ik}]|\mathbb{E}|X_{il}X_{it}^{r}|\leq\mathbb{E}|X_{ij}X_{ik}|\mathbb{E}|X_{il}X_{it}^{r}|\leq c_1'\mathbb{E}|X_{ij}|^{r+3}+c_2'\mathbb{E}|X_{ik}|^{r+3}+c_3'\mathbb{E}|X_{il}|^{r+3}+c_4'\mathbb{E}|X_{it}|^{r+3})$, where $c_1'+c_2'+c_3'+c_4'=1$.
 \hfill $\Box$

\vspace{2mm}

\noindent{\emph{Proof of Theorem \ref{thmsec2}.}} (i) Suppose that $\mathbf{g}$ is an even function. By a similar argument to the one used in the proof of Theorem \ref{theoremsec21}, 
\begin{align*}\sum_{j=1}^d\mathbb{E}\bigg[W_j\frac{\partial f}{\partial w_j}(\mathbf{W})\bigg]
=\frac{1}{n}\sum_{i=1}^n\sum_{j,k=1}^d\mathbb{E}[X_{ij}X_{ik}]\mathbb{E}\bigg[\frac{\partial^2f}{\partial w_j\partial w_k}(\mathbf{W})\bigg]+N_1+N_2+R_1 +R_2,
\end{align*}
where
\begin{align*}N_1&=\frac{1}{2n^{3/2}}\sum_{i=1}^n\sum_{j,k,l=1}^d\mathbb{E}[X_{ij}X_{ik}X_{il}]\mathbb{E}\bigg[\frac{\partial^3f}{\partial w_j\partial w_k\partial w_l}(\mathbf{W}^{(i)})\bigg],\\
N_2&=-\frac{1}{n^{3/2}}\sum_{i=1}^n\sum_{j,k,l=1}^d\mathbb{E}[X_{ij}X_{ik}]\mathbb{E}\bigg[X_{il}\frac{\partial^3f}{\partial w_j\partial w_k\partial w_l}(\mathbf{W})\bigg],\\
|R_1|&\leq \frac{1}{6n^2}\sum_{i=1}^n\sum_{j,k,l,t=1}^d\mathbb{E}\bigg|X_{ij}X_{ik}X_{il}X_{it}\frac{\partial^4f}{\partial w_j\partial w_k\partial w_l\partial w_t}(\mathbf{W}_{\theta_1}^{(i)})\bigg|, \\
|R_2|&\leq \frac{1}{2n^2}\sum_{i=1}^n\sum_{j,k,l,t=1}^d|\mathbb{E}[X_{ij}X_{ik}]|\mathbb{E}\bigg|X_{il}X_{it}\frac{\partial^4f}{\partial w_j\partial w_k\partial w_l\partial w_t}(\mathbf{W}_{\theta_2}^{(i)})\bigg|.
\end{align*}
We can write $N_1$ and $N_2$ as
\begin{align*}N_1&=\frac{1}{2n^{3/2}}\sum_{i=1}^n\sum_{j,k,l=1}^d\mathbb{E}[X_{ij}X_{ik}X_{il}]\mathbb{E}\bigg[\frac{\partial^3f}{\partial w_j\partial w_k\partial w_l}(\mathbf{W})\bigg]+R_3,\\
N_2&=-\frac{1}{n^{3/2}}\sum_{i=1}^n\sum_{j,k,l=1}^d\mathbb{E}[X_{ij}X_{ik}]\mathbb{E}[X_{il}]\mathbb{E}\bigg[\frac{\partial^3f}{\partial w_j\partial w_k\partial w_l}(\mathbf{W}^{(i)})\bigg]+R_4=R_4,
\end{align*}
where
\begin{align*}|R_3|&\leq \frac{1}{2n^2}\sum_{i=1}^n\sum_{j,k,l,t=1}^d|\mathbb{E}[X_{ij}X_{ik}X_{il}]|\mathbb{E}\bigg|X_{it}\frac{\partial^4f}{\partial w_j\partial w_k\partial w_l\partial w_t}(\mathbf{W}_{\theta_3}^{(i)})\bigg|,\\
|R_4|&\leq\frac{1}{n^2}\sum_{i=1}^n\sum_{j,k,l,t=1}^d|\mathbb{E}[X_{ij}X_{ik}]|\mathbb{E}\bigg|X_{il}X_{it}\frac{\partial^4f}{\partial w_j\partial w_k\partial w_l\partial w_t}(\mathbf{W}_{\theta_4}^{(i)})\bigg|.
\end{align*}
Combining bounds gives that
\begin{align}|\mathbb{E}[h(\mathbf{g}(\mathbf{W}))]-\mathbb{E}[h(\mathbf{g}(\mathbf{Z}))]|&\leq \frac{1}{2n^{3/2}}\sum_{i=1}^n\sum_{j,k,l=1}^d|\mathbb{E}[X_{ij}X_{ik}X_{il}]|\bigg|\mathbb{E}\bigg[\frac{\partial^3f}{\partial w_j\partial w_k\partial w_l}(\mathbf{W})\bigg]\bigg|\nonumber\\
\label{nearlyeven}&\quad+|R_1|+|R_2|+|R_3|+|R_4|.
\end{align}
To obtain the desired $O(n^{-1})$ bound we need to show that $\mathbb{E}[\frac{\partial^{3} f}{\partial w_j\partial w_k\partial w_l}(\mathbf{W})]$ is of order $n^{-1/2}$.
To do this, we consider the
Stein equation 
\begin{equation*}\nabla^\intercal\Sigma\nabla \psi_{jkl}(\mathbf{w})-\mathbf{w}^\intercal\nabla \psi_{jkl}(\mathbf{w})=\frac{\partial^{3} f}{\partial w_j\partial w_k\partial w_l}(\mathbf{w})-\mathbb{E}\bigg[\frac{\partial^{3} f}{\partial w_j\partial w_k\partial w_l}(\mathbf{Z})\bigg],
\end{equation*}
where $\mathbf{Z}\sim\mathrm{N}_d(\mathbf{0},\Sigma)$. Since $\mathbf{g}$ is an even function, the solution $f$, as given by (\ref{mvnsolnh}), is an even function (this is a simple generalisation of \cite[Lemma 3.2]{gaunt normal}).  Therefore $\mathbb{E}[\frac{\partial^{3} f}{\partial w_j\partial w_k\partial w_l}(\mathbf{Z})]=0$, and so
\begin{align}\mathbb{E}\bigg[\frac{\partial^{3} f}{\partial w_j\partial w_k\partial w_l}(\mathbf{W})\bigg]&=\mathbb{E}[\nabla^\intercal\Sigma\nabla \psi_{jkl}(\mathbf{W})-\mathbf{W}^\intercal\nabla \psi_{jkl}(\mathbf{W})]\nonumber\\
&\leq \frac{1}{n^{3/2}}\sum_{\alpha=1}^n\sum_{a,b,c=1}^d\bigg\{\frac{1}{2}\mathbb{E}\bigg|X_{\alpha a}X_{\alpha b}X_{\alpha c}\frac{\partial^3\psi_{jkl}}{\partial w_a\partial w_b\partial w_c}(\mathbf{W}_{\theta_5}^{(\alpha)})\bigg| \nonumber\\
\label{psiex}&\quad+|\mathbb{E}[X_{\alpha a}X_{\alpha b}]|\mathbb{E}\bigg|X_{\alpha c}\frac{\partial^3\psi_{jkl}}{\partial w_a\partial w_b\partial w_c}(\mathbf{W}_{\theta_6}^{(\alpha)})\bigg|\bigg\},
\end{align}
where we obtained the inequality by arguing as we did in the proof of Theorem \ref{theoremsec21}. Bounding $|R_1|$, $|R_2|$, $|R_3|$, $|R_4|$ and (\ref{psiex}) using the inequalities of Lemma \ref{cbwbshc} and applying these bounds to (\ref{nearlyeven}) now yields the bound 
\[|\mathbb{E}[h(\mathbf{g}(\mathbf{W}))]-\mathbb{E}[h(\mathbf{g}(\mathbf{Z}))]|\leq K_1'+K_2',
\]
where
\begin{align}K_1'=&\frac{h_{4,m}}{24n^2}\sum_{i=1}^n\sum_{j,k,l,t=1}^d\bigg\{A\mathbb{E}|X_{ij}X_{ik}X_{il}X_{it}|+B\sum_{q=1}^d \big(2^{3r/2}\mathbb{E}|X_{ij}X_{ik}X_{il}X_{it}|\mathbb{E}|W_q|^{r}\nonumber\\
&+\mathbb{E}|X_{ij}X_{ik}X_{il}X_{it}X_{iq}^{r}|+2^{r/2}\mu_{r,\sigma_q}\mathbb{E}|X_{ij}X_{ik}X_{il}X_{it}|\big)+9|\mathbb{E}[X_{ij}X_{ik}]|\bigg[A\mathbb{E}|X_{il}X_{it}| \nonumber\\
&+B\sum_{q=1}^d \big(2^{3r/2}\mathbb{E}|X_{il}X_{it}|\mathbb{E}|W_q|^{r}+\mathbb{E}|X_{il}X_{it}X_{iq}^{r}|+2^{r/2}\mu_{r,\sigma_q}\mathbb{E}|X_{il}X_{it}|\big)\bigg]\nonumber\\
&+3|\mathbb{E}[X_{ij}X_{ik}X_{il}]|\bigg[A\mathbb{E}|X_{it}|+B\sum_{q=1}^d \big(2^{3r/2}\mathbb{E}|X_{it}|\mathbb{E}|W_q|^{r}+\mathbb{E}|X_{it}X_{iq}^{r}|\nonumber \\
&+2^{r/2}\mu_{r,\sigma_q}\mathbb{E}|X_{it}|\big)\bigg]\bigg\},\nonumber\\
K_2'=&\frac{h_{6,m}}{72n^3}\sum_{i=1}^n\sum_{j,k,l=1}^d|\mathbb{E}[X_{i j}X_{i k}X_{i l}]|\sum_{\alpha=1}^n\sum_{a,b,c=1}^d\bigg\{A\mathbb{E}|X_{\alpha a}X_{\alpha b}X_{\alpha c}|\nonumber\\
&+B\sum_{q=1}^d\big(12^{r/2}\mathbb{E}|X_{\alpha a}X_{\alpha b}X_{\alpha c}|\mathbb{E}|W_q|^{r}+\mathbb{E}|X_{\alpha a}X_{\alpha b}X_{\alpha c}X_{\alpha q}^{r}|\nonumber\\
&+2\cdot 3^{r/2}\mu_{r,\sigma_q}\mathbb{E}|X_{\alpha a}X_{\alpha b}X_{\alpha c}|\big) +2|\mathbb{E}[X_{\alpha a}X_{\alpha b}]|\bigg[A\mathbb{E}|X_{\alpha c}|\nonumber\\
&+B\sum_{q=1}^d \big(12^{r/2}\mathbb{E}|X_{\alpha c}|\mathbb{E}|W_q|^{r}+\mathbb{E}|X_{\alpha c}X_{\alpha q}^{r}|+2\cdot 3^{r/2}\mu_{r,\sigma_q}\mathbb{E}|X_{\alpha c}|\big)\bigg]\bigg\}. \nonumber
\end{align}
We now simplify the bound as we did in the proof of Theorem \ref{theoremsec21}.

\vspace{2mm} 

\noindent{(ii)} 
Observe that the bound (\ref{nearlyeven}) reduces to $|R_1|+|R_2|+|R_3|+|R_4|$ when $\mathbb{E}[X_{ij}X_{ik}X_{il}]=0$ for all $1\leq i\leq n$ and $1\leq j,k,l\leq d$.  Therefore, when we impose this condition, $K_2'=0$ and the final term in $K_1'$ involving a multiple of $\mathbb{E}[X_{ij}X_{ik}X_{il}]$ vanishes. Thus, we no longer require $\mathbf{g}$ to be an even function and in this case, the classes of functions for $\mathbf{g}$ and $h$ weaken to $\mathbf{g}\in C_P^{4,m}(\mathbb{R}^d)$ and $h\in C_b^4(\mathbb{R}^m)$. 
 \qed

\vspace{2mm}

\noindent{\emph{Proof of Theorems \ref{indep1} and \ref{indep2}.}} For Theorem \ref{indep1}, we set $d=m=1$ in the bounds for $|R_1|$ and $|R_2|$ given in the proof of Theorem \ref{theoremsec21} and bound the expectations using the univariate bounds of Lemma 3.6 of \cite{GauntSut} (with a simple generalisation from unit variance to general variance $\mathrm{Var}(W)=\sigma^2$). Similarly, for Theorem \ref{indep2}, we set $d=m=1$ in the bound that results from combining inequalities (\ref{nearlyeven}) and (\ref{psiex}) and  bound the expectations using the univariate bounds of Lemma 3.6 of \cite{GauntSut} (with the same minor generalisation). Some minor simplifications of the resulting bounds (using the assumptions $n\geq8$ and $n\geq12$) yields the desired bounds. \qed

\section{Proofs of Theorems \ref{thm2.1}--\ref{thm2.3univ}}\label{sec5}


\noindent{\emph{Proof of Theorem \ref{thm2.1}}.} Let $\mathbf{g}_{t,n}(\mathbf{w})=n^{t/2}(\mathbf{f}(n^{-1/2}\mathbf{w})-\mathbf{f}(\mathbf{0}))$ and let $\mathbf{W}=\sqrt{n}\overline{\mathbf{X}}$ so that $\mathbf{T}_{t,n}=\mathbf{g}_{t,n}(\mathbf{W}).$ Let $\mathbf{Z}=(Z_1,\ldots,Z_d)^\intercal\sim\mathrm{N}_d(\mathbf{0},\Sigma)$. Then, by the triangle inequality,
	\begin{align}
		\Delta_h(\mathbf{T}_{t,n},\mathbf{Y}_t)&\leq|\mathbb{E}[h(\mathbf{g}_{t,n}(\mathbf{W}))]-\mathbb{E}[h(\mathbf{g}_{t,n}(\mathbf{Z}))]|+|\mathbb{E}[h(\mathbf{g}_{t,n}(\mathbf{Z}))]-\mathbb{E}[h(\mathbf{Y}_t)]|\nonumber\\
		&:=R_1+R_2.\label{triangleineq}
	\end{align}
 We begin by bounding $R_2$. By Taylor expanding $\mathbf{f}(n^{-1/2}\mathbf{w})$ about $\mathbf{0}$, and using that $\frac{\partial\mathbf{f}}{\partial w_{i_1}}(\mathbf{0})=\cdots=\frac{\partial^{t-1}\mathbf{f}}{\prod_{j=1}^{t-1}\partial w_{i_j}}(\mathbf{0})=\mathbf{0}$ for all $1\leq i_1,\ldots,i_{t-1}\leq d$ (in the case $t\geq2$), we have
\begin{align}\mathbb{E}[h(\mathbf{g}_{t,n}(\mathbf{Z}))]
	=\mathbb{E}\bigg[h\bigg(\mathbf{Y}_t+\frac{1}{(t+1)!\sqrt{n}}\sum_{i_1,\ldots,i_{t+1}=1}^{d}Z_{i_1}\cdots Z_{i_{t+1}}\frac{\partial^{t+1}\mathbf{f}}{\partial w_{i_1} \cdots\partial w_{i_{t+1}}}\bigg(\frac{\theta\mathbf{Z}}{\sqrt{n}}\bigg)\bigg)\bigg],\nonumber
\end{align} 
for some $\theta\in(0,1)$. 
Then, by the mean value theorem and the triangle inequality,
\begin{align}
	R_2&\leq\frac{|h|_1}{(t+1)!\sqrt{n}}\sum_{k=1}^{m}\sum_{i_1,\ldots,i_{t+1}=1}^{d}\mathbb{E}\bigg|Z_{i_1}\cdots Z_{i_{t+1}}\frac{\partial^{t+1}f_{k}}{\partial w_{i_1} \cdots\partial w_{i_{t+1}}}\bigg(\frac{\theta\mathbf{Z}}{\sqrt{n}}\bigg)\bigg|\nonumber\\
	&\leq\frac{A_{t+1}m|h|_1}{(t+1)!\sqrt{n}}\mathbb{E}\bigg[\bigg(1+\sum_{j=1}^{d}\bigg|\frac{Z_j}{\sqrt{n}}\bigg|^{r_{t+1}}\bigg)\sum_{i_1,\ldots,i_{t+1}=1}^{d}|Z_{i_1}\cdots Z_{i_{t+1}}|\bigg]\nonumber\\
	&\leq\frac{A_{t+1}m|h|_1}{(t+1)!\sqrt{n}}\mathbb{E}\bigg[\bigg(1+\sum_{j=1}^{d}\bigg|\frac{Z_j}{\sqrt{n}}\bigg|^{r_{t+1}}\bigg)d^t\sum_{\ell=1}^{d}|Z_\ell|^{t+1}\bigg]\nonumber\\
	&\leq\frac{A_{t+1}md^{t}|h|_1}{(t+1)!\sqrt{n}}\sum_{j=1}^{d}\bigg(\mathbb{E}|Z_j|^{t+1}+\frac{d}{n^{r_{t+1}/2}}\mathbb{E}|Z_j|^{r_{t+1}+t+1}\bigg). \nonumber
\end{align}

We bound $R_1$ using the bound (\ref{genbd1}) of Theorem \ref{theoremsec21}. To apply that theorem, we need to find a dominating function $P$ such that $\mathbf{g}_{t,n}\in C_P^{3,m}(\mathbb{R}^d)$. Suppose that $\mathbf{f}\in\mathcal{C}_P^{1,2,m}(\mathbb{R}^d)$ for $t=1$, and $\mathbf{f}\in\mathcal{C}_P^{t,1,m}(\mathbb{R}^d)$ for $t\geq2$. Under these assumptions, we can show that a suitable dominating function is given by $P(\mathbf{w})=2C_{1,t}a_{n,d,r_t}^3d^{3t-4}(d+\sum_{i=1}^d|w_i|^{u_{1,t}})$, where the constants $C_{1,t}$ and $u_{1,t}$ are defined as in the statement of the theorem. With this dominating function, we would have the bound $R_1\leq h_{3,m}\mathcal{M}_{2,d}/\sqrt{n}$,
where $\mathcal{M}_{2,d}$ is defined as in the statement of the theorem. Summing up our bounds for $R_1$ and $R_2$ yields the desired bound (\ref{thm2.1bd}), and it now remains to show that $\mathbf{g}_{t,n}\in C_P^{3,m}(\mathbb{R}^d)$, with $P$ as defined above. This is verified in the SM. 
\qed

\vspace{2mm}

\noindent{\emph{Proof of Theorem \ref{thm2.2}.}} We will begin by bounding $R_2$ in (\ref{triangleineq}). By Taylor expanding $\mathbf{f}(n^{-1/2}\mathbf{w})$ about $\mathbf{0}$ and using that $\frac{\partial\mathbf{f}}{\partial w_{i_1}}(\mathbf{0})=\cdots=\frac{\partial^{t-1}\mathbf{f}}{\prod_{j=1}^{t-1}\partial w_{i_j}}(\mathbf{0})=\mathbf{0}$ for all $1\leq i_1,\ldots, i_{t-1}\leq d$ (in the case $t\geq2$), we have
\begin{align} \mathbb{E}[h(\mathbf{g}_{t,n}(\mathbf{Z}))]	&=\mathbb{E}\bigg[h\bigg(\mathbf{Y}_t+\frac{1}{(t+1)!\sqrt{n}}\sum_{i_1,\ldots,i_{t+1}=1}^{d}Z_{i_1}\cdots Z_{i_{t+1}}\frac{\partial^{t+1}\mathbf{f}}{\partial w_{i_1} \cdots\partial w_{i_{t+1}}}(\mathbf{0})\nonumber\\
	&\quad+\frac{1}{(t+2)!n}\sum_{i_1,\ldots,i_{t+2}=1}^{d}Z_{i_1}\cdots Z_{i_{t+2}}\frac{\partial^{t+2}\mathbf{f}}{\partial w_{i_1} \cdots\partial w_{i_{t+2}}}\bigg(\frac{\theta\mathbf{Z}}{\sqrt{n}}\bigg)\bigg)\bigg],\nonumber
\end{align}
for some $\theta\in(0,1)$. Then, by Taylor expanding $h$ about $\mathbf{Y}_t$, we have
\begin{align*}
	|R_2|=|\mathbb{E}[h(\mathbf{g}_{t,n}(\mathbf{Z}))]-\mathbb{E}[h(\mathbf{Y}_t)]|
	=|R_3+R_4|,
\end{align*}
where
\begin{align*}
	R_3&=\mathbb{E}\bigg[\sum_{k=1}^{m}\frac{\partial h}{\partial w_{k}}(\mathbf{Y}_t)\bigg(\frac{1}{\sqrt{n}(t+1)!}\sum_{i_1,\ldots,i_{t+1}=1}^{d}Z_{i_1}\cdots Z_{i_{t+1}}\frac{\partial^{t+1}f_{k}}{\partial w_{i_1} \cdots\partial w_{i_{t+1}}}(\mathbf{0})\\
	&\quad+\frac{1}{n(t+2)!}\sum_{i_1,\ldots,i_{t+2}=1}^{d}Z_{i_1}\cdots Z_{i_{t+2}}\frac{\partial^{t+2}f_{k}}{\partial w_{i_1} \cdots\partial w_{i_{t+2}}}\bigg(\frac{\theta\mathbf{Z}}{\sqrt{n}}\bigg)\bigg)\bigg], \\
	|R_4|&\leq\frac{1}{2}|h|_2\mathbb{E}\bigg[\bigg(\sum_{k=1}^{m}\bigg|\frac{1}{\sqrt{n}(t+1)!}\sum_{i_1,\ldots,i_{t+1}=1}^{d}Z_{i_1}\cdots Z_{i_{t+1}}\frac{\partial^{t+1}f_{k}}{\partial w_{i_1} \cdots\partial w_{i_{t+1}}}(\mathbf{0})\\
&\quad+\frac{1}{n(t+2)!}\sum_{i_1,\ldots,i_{t+2}=1}^{d}Z_{i_1}\cdots Z_{i_{t+2}}\frac{\partial^{t+2}f_{k}}{\partial w_{i_1} \cdots\partial w_{i_{t+2}}}\bigg(\frac{\theta\mathbf{Z}}{\sqrt{n}}\bigg)\bigg|\bigg)^2\bigg].
\end{align*}

We note that
\begin{align*}
	\mathbb{E}\bigg[\sum_{k=1}^{m}\frac{\partial h}{\partial w_{k}}(\mathbf{Y}_t)\bigg(\frac{1}{\sqrt{n}(t+1)!}\sum_{i_1,\ldots,i_{t+1}=1}^{d}Z_{i_1}\cdots Z_{i_{t+1}}\frac{\partial^{t+1}f_{k}}{\partial w_{i_1} \cdots\partial w_{i_{t+1}}}(\mathbf{0})\bigg)\bigg]=0,
\end{align*}
since the expectation is of the form $\mathbb{E}[\tau(\mathbf{Z})]$, where $\tau$ is an odd function and $\mathbf{Z}\sim \mathrm{N}_d(\mathbf{0},\Sigma)$ is distributed symmetrically around the origin (that is $\mathbf{Z}=_d-\mathbf{Z}$), meaning the expectation must be equal to zero. Therefore
\begin{align}
	|R_3|&=\frac{1}{(t+2)!n}\bigg|\mathbb{E}\bigg[\sum_{k=1}^{m}\frac{\partial h}{\partial w_{k}}(\mathbf{Y}_t)\sum_{i_1,\ldots,i_{t+2}=1}^{d}Z_{i_1}\cdots Z_{i_{t+2}}\frac{\partial^{t+2}f_{k}}{\partial w_{i_1} \cdots\partial w_{i_{t+2}}}\bigg(\frac{\theta\mathbf{Z}}{\sqrt{n}}\bigg)\bigg]\bigg|\nonumber\\
	&\leq\frac{|h|_1}{(t+2)!n}\sum_{k=1}^{m}\sum_{i_1,\ldots,i_{t+2}=1}^{d}\mathbb{E}\bigg|Z_{i_1}\cdots Z_{i_{t+2}}\frac{\partial^{t+2}f_{k}}{\partial w_{i_1} \cdots\partial w_{i_{t+2}}}\bigg(\frac{\theta\mathbf{Z}}{\sqrt{n}}\bigg)\bigg|\nonumber\\
	&\leq\frac{A_{t+2}m|h|_1}{(t+2)!n}\mathbb{E}\bigg[\bigg(1+\sum_{j=1}^{d}\bigg|\frac{Z_j}{\sqrt{n}}\bigg|^{r_{t+2}}\bigg)\sum_{i_1,\ldots,i_{t+2}=1}^{d}\big|Z_{i_1}\cdots Z_{i_{t+2}}\big|\bigg]\nonumber\\
	&\leq\frac{A_{t+2}md^{t+1}|h|_1}{(t+2)!n}\sum_{j=1}^{d}\bigg(\mathbb{E}|Z_j|^{t+2}+\frac{d}{n^{r_{t+2}/2}}\mathbb{E}|Z_j|^{r_{t+2}+t+2}\bigg).\nonumber
\end{align}
Now we bound $R_4$:
\begin{align}
	|R_4|&\leq\frac{m^2|h|_2}{2((t+1)!)^2n}\mathbb{E}\bigg[\bigg(A_{t+1}\sum_{i_1,\ldots,i_{t+1}=1}^{d}\big|Z_{i_1}\cdots Z_{i_{t+1}}\big|\nonumber\\
	&\quad+\frac{A_{t+2}}{(t+2)\sqrt{n}}\bigg(1+\sum_{j=1}^{d}|Z_j|^{r_{t+2}}\bigg)\sum_{i_1,\ldots,i_{t+2}=1}^{d}\big|Z_{i_1}\cdots Z_{i_{t+2}}\big|\bigg)^2\bigg]\nonumber\\
	&\leq\frac{m^2|h|_2}{2((t+1)!)^2n}\mathbb{E}\bigg[\bigg(A_{t+1}d^t\sum_{j=1}^{d}|Z_j|^{t+1}\nonumber\\
	&\quad+\frac{A_{t+2}d^{t+1}}{(t+2)\sqrt{n}}\bigg(\sum_{j=1}^{d}|Z_j|^{t+2}+d\sum_{j=1}^{d}|Z_j|^{r_{t+2}+t+2}\bigg)\bigg)^2\bigg]\nonumber\\
	&\leq\frac{m^2d^{2t+1}|h|_2}{((t+1)!)^2n}\sum_{j=1}^{d}\bigg[A_{t+1}^2\mathbb{E}|Z_j|^{2(t+1)}+\frac{2A_{t+2}^2d^2}{(t+2)^2n}\bigg(\mathbb{E}|Z_j|^{2(t+2)}+d^{2}\mathbb{E}|Z_j|^{2(r_{t+2}+t+2)}\bigg)\bigg]\nonumber.
\end{align}

We now bound $R_1$ from (\ref{triangleineq}) using the bound (\ref{boundthm2}) of part (i) of Theorem \ref{thmsec2}. To apply that bound, we need to find a dominating function $P$ such that $\mathbf{g}_{t,n}\in C^{6,m}_P(\mathbb{R}^d)$. Suppose that $\mathbf{f}\in\mathcal{C}_P^{2,4,m}(\mathbb{R}^d)$ for $t=2$ and $\mathbf{f}\in\mathcal{C}_P^{t,2,m}(\mathbb{R}^d)$ for $t\geq4$. Under these assumptions,  a suitable dominating function is given by $P(\mathbf{w})=2{C}_{2,t}a_{n,d,r_t}^6d^{6t-7}(d+\sum_{i=1}^{d}|w_i|^{u_{2,t}}),$ where, for $t\geq2$, the constants ${C}_{2,t}$ and $u_{2,t}$ are defined as in the statement of the theorem. This is verified in the SM. Applying the bound (\ref{boundthm2}) with this dominating function yields the bound
$R_1\leq (h_{4,m}\mathcal{K}_{3,d}+h_{6,m}\mathcal{K}_{4,d})/n$.
Summing up our bounds for $R_1$, $R_3$ and $R_4$ yields the desired bound (\ref{thm2.2bd}).
\qed

\vspace{2mm}
\noindent{\emph{Proof of Theorem \ref{thm2.3}.}} We have the same bound for $R_2$ as in Theorem \ref{thm2.2}, since the assumption that $\mathbf{f}$ is an even function is not used in obtaining this bound.
To bound $R_1$, we use the bound (\ref{boundtheorem}) from part (ii) of Theorem \ref{thmsec2}. To apply this theorem, we need to find a dominating function $P$ such that $\mathbf{g}_{t,n}\in C_P^{4,m}(\mathbb{R}^d).$ Suppose that $\mathbf{f}\in \mathcal{C}^{2,2,m}_P(\mathbb{R}^d)$ for $t=2$ and $\mathbf{f}\in \mathcal{C}_P^{t,2,m}(\mathbb{R}^d)$ for $t\geq4$. Under these assumptions, a suitable dominating function is given by $P(\mathbf{w})=2{C}_{3,t}a_{n,d,r_t}^4d^{4t-5}(d+\sum_{i=1}^{d}|w_i|^{u_{3,t}}),$ where ${C}_{3,t}$ and $u_{3,t}$ are defined as in the statement of the theorem. This is verified in the SM. With this dominating function, we obtain the bound $R_1\leq h_{4,m}\mathcal{K}_{5,d}/n$.
Summing up our bounds for $R_1$ and $R_2$ yields the bound (\ref{thm2.3bound}).
\qed

\vspace{2mm}

\noindent{\emph{Proof of Theorem \ref{thm2.1univ}.}} As in the multivariate case, we bound $R_1$ and $R_2$ from (\ref{triangleineq}). We can immediately bound $R_2$ by setting $d=m=1$ in the bound $m|h|_1\mathcal{M}_{1,d}/\sqrt{n}$. We bound $R_1$ by using the bound (\ref{a2bound}) of Theorem \ref{indep1}. To apply that bound, we need to find a dominating function $P$ such that $g_{t,n}\in C^{1,1}_{P}(\mathbb{R})$. For $f\in\mathcal{C}_P^{t,1,1}(\mathbb{R})$, a suitable dominating function is given by $P(w)=2A_t(1+|w|^{r_t+t-1})/(t-1)!$. This is verified in the SM. Bounding $R_1$ using inequality (\ref{a2bound}) with this dominating function yields the bound (\ref{thm2.4bd}). \qed

\vspace{2mm}

\noindent{\emph{Proof of Theorem \ref{thm2.2univ}.}} As in the proof of Theorem \ref{thm2.1univ}, we can immediately bound $R_2$ by $(\|h'\|\mathcal{K}_{1,1}+\|h''\|\mathcal{K}_{2,1})/n$. We bound $R_1$ by applying inequality (\ref{a.71}) from part (i) of Theorem \ref{indep2} with dominating function $P(w)=2{C}_{4,t}(1+|w|^{2(r_t+t-1)})$, which ensures that $g_{t,n}\in C_P^{2,1}(\mathbb{R})$ for $f\in\mathcal{C}_P^{t,2,1}(\mathbb{R})$; see the SM for a verification. Summing up our bounds for $R_1$ and $R_2$ yields the desired bound. \qed

\vspace{2mm}

\noindent{\emph{Proof of Theorem \ref{thm2.3univ}.}} We have the same bound for $R_2$ as in Theorem \ref{thm2.2univ}, since the assumption that $f$ is an even function is not used in obtaining this bound. We also use the same dominating function, $P(w)=2{C}_{4,t}(1+|w|^{2(r_t+t-1)})$, but we instead use the bound (\ref{a.72}) of part (ii) of Theorem \ref{indep2}  to bound $R_1$. Combining the bounds for $R_1$ and $R_2$ yields the desired bound. \qed
		
		\section*{Acknowledgements}
		 HS is supported by an EPSRC PhD Studentship.

\appendix

\section{Further proofs and calculations}\label{appa}

\subsection{Calculation of the covariance matrix from Example \ref{ex3.5}}

For fixed $j=1,\ldots,r$, $\pi_1(j),\ldots,\pi_n(j)$ are independent $\mathrm{Unif}\{1,\ldots,r\}$ random variables, and so we have
\begin{align*}\sigma_{jj}=\mathrm{Var}(W_j)=\frac{1}{\sigma_J^2n}\sum_{i=1}^n\mathrm{Var}(J(\pi_i(j))=\frac{1}{\sigma_J^2}\sum_{k=1}^r(J(k)-\overline{J})^2\frac{1}{r}=\frac{r-1}{r}.
\end{align*}
Now suppose $j\not=k$. As $\sum_{j=1}^rW_j=0$ and $W_1,\ldots,W_r$ are identically distributed, we have
\begin{equation*}0=\mathbb{E}\bigg[W_j\sum_{l=1}^rW_l\bigg]=\mathbb{E}[W_j^2]+\sum_{j\not=l}\mathbb{E}[W_jW_l]=\mathbb{E}[W_j^2]+(r-1)\mathbb{E}[W_jW_k].
\end{equation*}
Since $\mathbb{E}[W_j]=(r-1)/r$, rearranging yields $\mathbb{E}[W_jW_k]=-1/r$ for $j\not=k$. Since $\mathbb{E}[W_j]=0$, it therefore follows that $\sigma_{jk}=\mathrm{Cov}(W_j,W_k)=\mathbb{E}[W_jW_k]=-1/r$. \qed

\subsection{Proof of Lemma \ref{cbwbshc}}\label{applempf}

We first state the following lemma which contains bounds given in \cite{GauntSut}. We have generalised these bounds from the $m=1$ case to general $m\geq1$ by applying Lemma \ref{bell lem} and the considerations of Remark \ref{remm1}. 

\begin{lemma}[Gaunt and Sutcliffe \cite{GauntSut}, Propositions 2.1 and 2.2]\label{cor28} Let $P(\mathbf{w})=A+B\sum_{i=1}^d|w_i|^{r}$, where $r\geq 0$.  Let $\sigma_{i,i}=(\Sigma)_{i,i}$, $i=1,\ldots,d$. Let $f(=f_h)$ denote the solution (\ref{mvnsolnh}) of the multivariate normal Stein equation (\ref{mvng}), and let $\psi_{jkl}$ denote the solution of (\ref{234multinor}).  Suppose that $\Sigma$ is non-negative definite.

\vspace{2mm}

\noindent{(i)} Let $h\in C_b^t(\mathbb{R}^m)$ and $g\in C_P^{t,m}(\mathbb{R}^d)$ for $t\geq 1$.  Then, for $\mathbf{w}\in\mathbb{R}^d$,
\begin{align}\label{sigmaiii}\bigg|\frac{\partial^tf(\mathbf{w})}{\prod_{j=1}^t\partial w_{i_j}}\bigg|&\leq\frac{h_{t,m}}{t}\bigg[A+\sum_{i=1}^d2^{r/2}B\big(|w_i|^{r}+\mu_{r,\sigma_i}\big)\bigg].
\end{align}
\noindent{(ii)} Let $h\in C_b^{6}(\mathbb{R}^m)$ and $g\in C_P^{6,m}(\mathbb{R}^d)$.  Then, for $\mathbf{w}\in\mathbb{R}^d$,
\begin{align}\label{sigmaiii2}\bigg|\frac{\partial^3\psi_{jkl}(\mathbf{w})}{\partial w_a\partial w_b\partial w_c}\bigg|&\leq\frac{h_{6,m}}{18}\bigg[A+\sum_{i=1}^d3^{r/2}B\big(|w_i|^{r}+2\mu_{r,\sigma_i}\big)\bigg].
\end{align}
\end{lemma} 

\noindent{\emph{Proof of Lemma \ref{cbwbshc}.}} Let us prove the first inequality.  From inequality (\ref{sigmaiii}) we have
\begin{align*}\mathbb{E}\bigg|u(\mathbf{X}_i)\frac{\partial^tf}{\prod_{j=1}^{t}\partial w_{i_j}}(\mathbf{W}_{\theta}^{(i)})\bigg| &\leq \frac{h_{t,m}}{t}\bigg[A\mathbb{E}[u(\mathbf{X}_i)]+\sum_{j=1}^d2^{r/2}B\Big(\mathbb{E}|u(\mathbf{X}_i)(W_{j,\theta}^{(i)})^{r}|\\
&\quad+\mu_{r,\sigma_j}\mathbb{E}|Z|^{r}\mathbb{E}[u(\mathbf{X}_i)]\Big)\bigg],
\end{align*}
where $W_{j,\theta}^{(i)}$ is the $j$-th component of $\mathbf{W}_{\theta}^{(i)}$.   By using the crude inequality $|a+b|^s\leq 2^s(|a|^s+|b|^s)$, which holds for any $s\geq0$, and independence of $X_{ij}$ and $W_j^{(i)}$, we have 
\begin{align*}2^{r/2}\mathbb{E}|u(\mathbf{X}_i)(W_{j,\theta}^{(i)})^{r}|&\leq 2^{3r/2}\mathbb{E}\bigg|u(\mathbf{X}_i)\bigg(|W_j^{(i)}|^{r}+\frac{\theta^{r}}{n^{r/2}}|X_{ij}|^{r}\bigg)\bigg| \\
&\leq 2^{3r/2}\mathbb{E}[u(\mathbf{X}_i)]\mathbb{E}|W_j^{(i)}|^{r}+\mathbb{E}|X_{ij}^{r}u(\mathbf{X}_i)|,
\end{align*}
where we made use of the assumption that  $n\geq8$ to simplify the bound. Using that $\mathbb{E}|W_j^{(i)}|^{r}\leq \mathbb{E}|W_j|^{r}$ (see \cite[p.\ 1501]{gaunt normal}) leads to the first inequality.  The proof of the other inequality is similar; we just use inequality (\ref{sigmaiii2}) instead of inequality (\ref{sigmaiii}). We also simplify the bound by using the assumption $n\geq12$, so that $3^{r/2}\cdot 2^r/n^{r/2}=12^{r/2}/n^{r/2}\leq1$.\qed

\subsection{Verification of dominating functions}\label{supb}

\noindent{\bf{Verification for the proof of Theorem \ref{thm2.1}.}} 
To ease notation, we prove the result under the assumption that $d/n^{r_t/2}\leq1$, for which $a_{d,n,r_t}=1$; it is simple to see that in the general case $n,d\geq1$, $r_t\geq0$ picking up an additional factor of $a_{n,d,r_t}^3$ yields the required dominating function.    
We consider the cases $t=1$, $t=2$ and $t\geq 3$ separately. We first consider the case $t\geq3$. By Taylor expanding $\frac{\partial\mathbf{f}}{\partial w_{i_1}}(n^{-1/2}\mathbf{w})$ about $\mathbf{0}$, and using that $\frac{\partial\mathbf{f}}{\partial w_{i_1}}(\mathbf{0})=\cdots=\frac{\partial^{t-1}\mathbf{f}}{\prod_{j=1}^{t-1}\partial w_{i_j}}(\mathbf{0})=\mathbf{0}$ for all $1\leq i_1,\ldots,i_{t-1}\leq d$ (in the case $t\geq2$), we obtain that
\begin{align*}
	\frac{\partial\mathbf{g}_{t,n}}{\partial w_{i_1}}(\mathbf{w})&=\frac{1}{(t-1)!}\sum_{i_2,\ldots,i_{t}=1}^{d}w_{i_2}\cdots w_{i_{t}}\frac{\partial^{t}\mathbf{f}}{\partial w_{i_1}\cdots\partial w_{i_{t}}}\bigg(\frac{\theta\mathbf{w}}{\sqrt{n}}\bigg),
\end{align*}
 for some $\theta\in(0,1)$ (which may change from line to line).  Since $\mathbf{f}\in\mathcal{C}_P^{t,1,m}(\mathbb{R}^d)$, we obtain the bound
\begin{align}
	\bigg|\frac{\partial g_{t,j}}{\partial w_{i_1}}(\mathbf{w})\bigg|	&\leq \frac{A_t}{(t-1)!}\bigg(1+\sum_{i=1}^{d}\bigg|\frac{w_i}{\sqrt{n}}\bigg|^{r_t}\bigg)\sum_{i_2,\ldots,i_t=1}^{d}|w_{i_2}\cdots w_{i_t}|\nonumber\\
	&\leq \frac{A_td^{t-2}}{(t-1)!}\sum_{i=1}^{d}\bigg(|w_i|^{t-1}+\frac{d|w_i|^{r_t+t-1}}{n^{r_t/2}}\bigg)\nonumber \\
	&\leq \frac{A_td^{t-2}}{(t-1)!}\sum_{i=1}^{d}\big(|w_i|^{t-1}+|w_i|^{r_t+t-1}\big), \label{2.11bound}
\end{align}
where $g_{t,j}$ represents the $j$-th component of the function $\mathbf{g}_{t,n}$, and in obtaining the last step we used the assumption that $d/n^{r_t/2}\leq1$. Proceeding similarly, we obtain
\begin{align}
	\bigg|\frac{\partial^2g_{t,j}}{\partial w_{i_1}\partial w_{i_2}}(\mathbf{w})\bigg|&=\frac{1}{(t-2)!}\bigg|\sum_{i_3,\ldots,i_{t}=1}^{d}w_{i_3}\cdots w_{i_{t}}\frac{\partial^tf_j}{\partial w_{i_1}\cdots\partial w_{i_{t}}}\bigg(\frac{\theta\mathbf{w}}{\sqrt{n}}\bigg)\bigg|\nonumber\\
	&\leq \frac{A_t}{(t-2)!}\bigg(1+\sum_{i=1}^{d}\bigg|\frac{w_i}{\sqrt{n}}\bigg|^{r_t}\bigg)\sum_{i_3,\ldots,i_t=1}^{d}|w_{i_3}\cdots w_{i_{t}}|\nonumber\\
&\leq \frac{A_td^{t-3}}{(t-2)!}\sum_{i=1}^{d}\big(|w_i|^{t-2}+d|w_i|^{r_t+t-2}\big).\label{2.12bound}	
\end{align}
We deal with the third order partial derivatives similarly:
\begin{align}
	\bigg|\frac{\partial^3g_{t,j}}{\partial w_{i_1}\partial w_{i_2}\partial w_{i_3}}(\mathbf{w})\bigg|&\leq \frac{1}{(t-3)!}\bigg|\sum_{i_4,\ldots,i_{t}=1}^{d}w_{i_4}\cdots w_{i_{t}}\frac{\partial^tf_j}{\partial w_{i_1}\cdots\partial w_{i_{t}}}\bigg(\frac{\theta\mathbf{w}}{\sqrt{n}}\bigg)\bigg|\nonumber\\
	&\leq \frac{A_t}{(t-3)!}\bigg(1+\sum_{i=1}^{d}\bigg|\frac{w_i}{\sqrt{n}}\bigg|^{r_t}\bigg)\sum_{i_4,\ldots,i_t=1}^{d}|w_{i_4}\cdots w_{i_{t}}|\nonumber\\
	&\leq \frac{A_td^{t-4}}{(t-3)!}\sum_{i=1}^{d}\big(|w_i|^{t-3}+d|w_i|^{r_t+t-3}\big)\label{2.13bound}\\
	&\leq 2C_{1,t}d^{3t-4}\bigg(d+\sum_{i=1}^{d}|w_i|^{u_{1,t}}\bigg),\label{deriv3bound}
\end{align}
where in final step we used the basic inequalities $|w_i|^{t-3}\leq 1+|w_i|^{u_{1,t}}$ and $|w_i|^{r_t+t-3}\leq1+|w_i|^{u_{1,t}}$, since $t-3\leq r_t+t-3\leq u_{1,t}$.
Applying the inequality $(|a_1|+\ldots+|a_q|)^s\leq q^{s-1}(|a_1|^s+\ldots+|a_q|^s)$, $q,s\geq1$, to the bound (\ref{2.11bound}) yields the bound
\begin{align}\bigg|\frac{\partial g_{t,j}}{\partial w_{i_1}}(\mathbf{w})\bigg|^3\leq &\frac{4A_t^3d^{3t-4}}{((t-1)!)^3}\sum_{i=1}^{d}\big(|w_i|^{3t-3}+|w_i|^{3r_t+3t-3}\big)\leq  \frac{8A_t^3d^{3t-4}}{((t-1)!)^3}\bigg(d+\sum_{i=1}^{d}|w_i|^{u_{1,t}}\bigg).\label{cubed}
\end{align}
Similarly, from (\ref{2.12bound}), we get
\begin{align}\bigg|\frac{\partial^2g_{t,j}}{\partial w_{i_1}\partial w_{i_2}}(\mathbf{w})\bigg|^{3/2}&\leq \frac{\sqrt{2}A_t^{3/2}d^{3t/2-{4}}}{((t-2)!)^{3/2}}\sum_{i=1}^{d}\big(|w_i|^{3t/2-3/2}+{d^{3/2}}|w_i|^{3(r_t+t-2)/2}\big)\nonumber\\
&\leq  \frac{2\sqrt{2}A_t^{3/2}d^{3t-4}}{((t-2)!)^{3/2}}\bigg(d+\sum_{i=1}^{d}|w_i|^{u_{1,t}}\bigg).\label{3over2}
\end{align}
From inequalities (\ref{deriv3bound}), (\ref{cubed}) and (\ref{3over2}) we deduce that $\mathbf{g}_{t,n}\in C_P^{3,m}(\mathbb{R}^d)$ for $t\geq3$.


We now consider the case $t=2$. Inequality (\ref{cubed}) still applies for $t=2$, but inequalities (\ref{deriv3bound}) and (\ref{3over2}) were derived under the assumption that $t\geq3$.
We have $\mathbf{g}_{2,n}(\mathbf{w})=n(\mathbf{f}(n^{-1/2}\mathbf{w})-\mathbf{f}(\mathbf{0}))$, and so
\begin{equation*}
	\bigg|\frac{\partial^2 g_{2,j}}{\partial w_{i_1}\partial w_{i_2}}(\mathbf{w})\bigg|=\bigg|\frac{\partial ^2f_j}{\partial w_{i_1}\partial w_{i_2}}\bigg(\frac{\mathbf{w}}{\sqrt{n}}\bigg)\bigg|\leq A_2\bigg(1+\sum_{i=1}^{d}\bigg|\frac{w_i}{\sqrt{n}}\bigg|^{r_2}\bigg)\leq A_2\bigg(1+\sum_{i=1}^{d}|w_i|^{r_2}\bigg),
\end{equation*}
from which we conclude that inequality (\ref{2.12bound}) also holds for $t=2$, which in turn implies that the bound (\ref{3over2}) is valid for $t=2$. We also have the following bound for the third order partial derivatives: 
\begin{align}\label{2third}
	\bigg|\frac{\partial^3 g_{2,j}}{\partial w_{i_1}\partial w_{i_2}\partial w_{i_3}}(\mathbf{w})\bigg|&=\bigg|\frac{1}{\sqrt{n}}\frac{\partial^3 f_{j}}{\partial w_{i_1}\partial w_{i_2}\partial w_{i_3}}\bigg(\frac{\mathbf{w}}{\sqrt{n}}\bigg)\bigg|\leq\frac{A_3}{\sqrt{n}}\bigg(1+\sum_{i=1}^{d}|w_i|^{r_3}\bigg).
\end{align}
 From inequalities (\ref{cubed}), (\ref{3over2}) and (\ref{2third}), and similar considerations to those used for the case $t\geq3$, we conclude that $\mathbf{g}_{2,n}\in C_P^{3,m}(\mathbb{R}^d)$.

Finally, we consider the case $t=1$. 
Similarly to how we showed that inequality (\ref{3over2}) was also valid for $t=2$, one can show that inequality (\ref{cubed}) also holds for $t=1$; we omit the details.
Differentiating $\mathbf{g}_{1,n}(\mathbf{w})=\sqrt{n}(\mathbf{f}(n^{-1/2}\mathbf{w})-\mathbf{f}(\mathbf{0}))$ and using that $\mathbf{f}\in \mathcal{C}_P^{1,2,m}(\mathbb{R}^d)$ leads to the following bounds for the second and third order partial derivatives:
\begin{align}
	\bigg|\frac{\partial^2 g_{1,j}}{\partial w_{i_1}\partial w_{i_2}}(\mathbf{w})\bigg|&=\bigg|\frac{1}{\sqrt{n}}\frac{\partial^2 f_{j}}{\partial w_{i_1}\partial w_{i_2}}\bigg(\frac{\mathbf{w}}{\sqrt{n}}\bigg)\bigg|\nonumber\\
	&\leq\frac{A_2}{\sqrt{n}}\bigg(1+\sum_{i=1}^{d}|w_i|^{r_2}\bigg)\leq\frac{A_2}{n^{1/3}}\bigg(1+\frac{1}{d}\sum_{i=1}^{d}|w_i|^{r_2}\bigg),\label{1two} \\
	\bigg|\frac{\partial^3 g_{1,j}}{\partial w_{i_1}\partial w_{i_2}\partial w_{i_3}}(\mathbf{w})\bigg|&=\bigg|\frac{1}{n}\frac{\partial^3 f_{j}}{\partial w_{i_1}\partial w_{i_2}\partial w_{i_3}}\bigg(\frac{\mathbf{w}}{\sqrt{n}}\bigg)\bigg|\nonumber\\
	&\leq\frac{A_3}{n}\bigg(1+\sum_{i=1}^{d}|w_i|^{r_3}\bigg)\leq\frac{A_3}{n^{5/6}}\bigg(1+\frac{1}{d}\sum_{i=1}^{d}|w_i|^{r_3}\bigg),\label{1three}
\end{align}
where we used the assumption that $n\geq d^6$. From inequalities (\ref{cubed}), (\ref{1two}) and (\ref{1three}), using similar considerations to those used throughout this proof, 
it readily follows that $\mathbf{g}_{1,n}\in C_P^{3,m}(\mathbb{R}^d)$. We have thus finally shown that $\mathbf{g}_{t,n}\in C_P^{3,m}(\mathbb{R}^d)$ for $t\geq1$, and the proof is now complete. \qed

\vspace{2mm}

\noindent{\bf{Verification for the proof of Theorem \ref{thm2.2}.}} As in the verification for the proof of Theorem \ref{thm2.1} we assume that $d/n^{r_t/2}\leq1$ to simplify notation with the general case easily following.	We first consider the case $t\geq6$. Applying the  inequality $\big(|a_1|+\cdots+|a_q|\big)^s\leq q^{s-1}\big(|a_1|^s+\cdots+|a_q|^s\big),$ for $q,s\geq1,$ to the bound (\ref{2.11bound}) yields
\begin{align}
	\bigg|\frac{\partial g_{t,j}}{\partial w_{i_1}}(\mathbf{w})\bigg|^6&\leq\frac{A_t^6}{((t-1)!)^6}d^{6(t-2)}\bigg(\sum_{i=1}^{d}\big(|w_i|^{t-1}+|w_i|^{r_t+t-1}\big)\bigg)^6\nonumber\\
	&\leq\frac{32A_t^6}{((t-1)!)^6}d^{6t-7}\sum_{i=1}^{d}\big(|w_i|^{6(t-1)}+|w_i|^{6(r_t+t-1)}\big)\nonumber\\
	&\leq2{C}_{2,t}d^{6t-7}\bigg(d+\sum_{i=1}^{d}|w_i|^{u_{2,t}}\bigg).\label{2.2dom1bound}
	\end{align}
By applying similar considerations to the bound (\ref{2.12bound}), we obtain
\begin{align}
	\bigg|\frac{\partial^2 g_{t,j}}{\partial w_{i_1}\partial w_{i_2}}(\mathbf{w})\bigg|^3&
	\leq\frac{A_t^3}{((t-2)!)^3}d^{3(t-3)}\bigg(\sum_{i=1}^{d}\big(|w_i|^{t-2}+d|w_i|^{r_t+t-2}\big)\bigg)^3\nonumber\\
	&\leq\frac{4A_t^3}{((t-2)!)^3}d^{3t-7}\sum_{i=1}^{d}\big(|w_i|^{3(t-2)}+d^3|w_i|^{3(r_t+t-2)}\big)\nonumber\\
	&\leq2{C}_{2,t}d^{6t-7}\bigg(d+\sum_{i=1}^{d}|w_i|^{u_{2,t}}\bigg).\label{2.2dom2bound}
\end{align}
Again, applying similar considerations to the bound (\ref{2.13bound}) gives
\begin{align}
		\bigg|\frac{\partial^3 g_{t,j}}{\partial w_{i_1}\partial w_{i_2}\partial w_{i_3}}(\mathbf{w})\bigg|^2&\leq\frac{A_t^2}{((t-3)!)^2}d^{2(t-4)}\bigg(\sum_{i=1}^{d}\big(|w_i|^{t-3}+d|w_i|^{r_t+t-3}\big)\bigg)^2\nonumber\\
		&\leq\frac{2A_t^2}{((t-3)!)^2}d^{2t-7}\bigg(|w_i|^{2(t-3)}+d^2|w_i|^{2(r_t+t-3)}\bigg)\nonumber\\
		&\leq2{C}_{2,t}d^{6t-7}\bigg(d+\sum_{i=1}^{d}|w_i|^{u_{2,t}}\bigg).\label{2.2dom3bound}
\end{align}
We now obtain bounds for the fourth, fifth and sixth order partial derivatives of $g_{t,j}$. By Taylor expanding $\frac{\partial^4\mathbf{f}}{\partial w_{i_1}\cdots\partial w_{i_4}}(n^{-1/2}\mathbf{w})$ about $\mathbf{0}$ and using $\frac{\partial^4\mathbf{f}}{\partial w_{i_1}\cdots\partial w_{i_4}}(\mathbf{0})=\cdots=\frac{\partial^{t-1}\mathbf{f}}{\prod_{j=1}^{t-1}\partial w_{i_j}}(\mathbf{0})=\mathbf{0}$ for all $1\leq i_1,\ldots, i_{t-1}\leq d$ we obtain that, for $t\geq5,$
\begin{align}
	\bigg|\frac{\partial^4 g_{t,j}}{\partial w_{i_1}\cdots\partial w_{i_4}}(\mathbf{w})\bigg|&\leq\frac{1}{(t-4)!}\sum_{i_5,\ldots,i_t=1}^{d}\bigg|w_{i_5}\cdots w_{i_t}\frac{\partial^tf_j}{\partial w_{i_1}\cdots\partial w_{i_t}}\bigg(\frac{\theta\mathbf{w}}{\sqrt{n}}\bigg)\bigg|\nonumber\\
	&\leq\frac{A_t}{(t-4)!}\bigg(1+\sum_{i=1}^{d}\bigg|\frac{w_i}{\sqrt{n}}\bigg|^{r_t}\bigg)\sum_{i_5,\ldots,i_t=1}^{d}|w_{i_5}\cdots w_{i_t}|\nonumber\\
	&\leq\frac{A_t}{(t-4)!}d^{t-5}\sum_{i=1}^{d}\big(|w_i|^{t-4}+d|w_i|^{r_t+t-4}\big).\label{2.24bound}
\end{align}
Note that inequality (\ref{2.24bound}) is also valid for $t=4$; this can be seen using a similar argument to one used in the proof of Theorem \ref{thm2.1} to show that inequality (\ref{3over2}) was valid for $t=2$.
 Proceeding similarly, we obtain, for $t\geq6$,
\begin{align}
	\bigg|\frac{\partial^5 g_{t,j}}{\partial w_{i_1}\cdots\partial w_{i_5}}(\mathbf{w})\bigg|
	&\leq\frac{A_t}{(t-5)!}d^{t-6}\sum_{i=1}^{d}\big(|w_i|^{t-5}+d|w_i|^{r_t+t-5}\big).\label{2.25bound}
\end{align}
We deal with the sixth order partial derivatives similarly:
\begin{align}
	\bigg|\frac{\partial^6 g_{t,j}}{\partial w_{i_1}\cdots\partial w_{i_6}}(\mathbf{w})\bigg|&\leq\frac{1}{(t-6)!}\sum_{i_7,\ldots,i_t=1}^{d}\bigg|w_{i_7}\cdots w_{i_t}\frac{\partial^tf_j}{\partial w_{i_1}\cdots\partial w_{i_t}}\bigg(\frac{\theta\mathbf{w}}{\sqrt{n}}\bigg)\bigg|\nonumber\\
	&\leq\frac{A_t}{(t-6)!}\bigg(1+\sum_{i=1}^{d}\bigg|\frac{w_i}{\sqrt{n}}\bigg|^{r_t}\bigg)\sum_{i_7,\ldots,i_t=1}^{d}|w_{i_7}\cdots w_{i_t}|\nonumber\\
	&\leq\frac{A_t}{(t-6)!}d^{t-7}\sum_{i=1}^{d}\big(|w_i|^{t-6}+d|w_i|^{r_t+t-6}\big)\nonumber\\
	&\leq2{C}_{2,t}d^{6t-7}\bigg(d+\sum_{i=1}^{d}|w_i|^{u_{2,t}}\bigg),\label{2.2dom6bound}
\end{align}
for some $\theta\in(0,1)$, again noting that the bound (\ref{2.2dom6bound}) is valid for $t\geq6$. Applying the basic inequality $\big(|a_1|+\cdots+|a_q|\big)^s\leq q^{s-1}\big(|a_1|^s+\cdots+|a_q|^s\big),$ for $q,s\geq1,$ to inequality (\ref{2.24bound}) yields the bound
\begin{align}
		\bigg|\frac{\partial^4 g_{t,j}}{\partial w_{i_1}\cdots\partial w_{i_4}}(\mathbf{w})\bigg|^{3/2}&\leq\frac{\sqrt{2}A_t^{3/2}}{((t-4)!)^{3/2}}d^{3t/2-7}\sum_{i=1}^{d}\big(|w_i|^{3(t-4)/2}+d^{3/2}|w_i|^{3(r_t+t-4)/2}\big)\nonumber\\
		&\leq2{C}_{2,t}d^{6t-7}\bigg(d+\sum_{i=1}^{d}|w_i|^{u_{2,t}}\bigg).\label{2.2dom4bound}
\end{align}
Similarly, from (\ref{2.25bound}), we have
\begin{align}
	\bigg|\frac{\partial^5 g_{t,j}}{\partial w_{i_1}\cdots\partial w_{i_5}}(\mathbf{w})\bigg|^{6/5}&\leq\frac{2^{1/5}A_t^{6/5}}{((t-5)!)^{6/5}}d^{6t/5-7}\sum_{i=1}^{d}\big(|w_i|^{6(t-5)/5}+d^{6/5}|w_i|^{6(r_t+t-5)/5}\big)\nonumber\\
	&\leq2{C}_{2,t}d^{6t-7}\bigg(d+\sum_{i=1}^{d}|w_i|^{u_{2,t}}\bigg).\label{2.2dom5bound}
\end{align}
From (\ref{2.2dom1bound}), (\ref{2.2dom2bound}), (\ref{2.2dom3bound}), (\ref{2.2dom6bound}), (\ref{2.2dom4bound}) and (\ref{2.2dom5bound}), we deduce that $\mathbf{g}_{t,n}\in C_p^{6,m}(\mathbb{R}^d),$ for $t\geq6$.

We now consider the case $t=4$. The inequalities (\ref{2.2dom1bound}), (\ref{2.2dom2bound}), (\ref{2.2dom3bound}) and (\ref{2.2dom4bound}) all hold for $t=4$, but inequalities (\ref{2.2dom6bound}) and (\ref{2.2dom5bound}) were derived under the assumption that $t\geq6$. We have $\mathbf{g}_{4,n}(\mathbf{w})=n^2(\mathbf{f}(n^{-1/2}\mathbf{w})-\mathbf{f}(\mathbf{0}))$, and so 
\begin{align*}
	\bigg|\frac{\partial^5 g_{4,j}}{\partial w_{i_1}\cdots\partial w_{i_5}}(\mathbf{w})\bigg|\leq\bigg|\frac{1}{\sqrt{n}}\frac{\partial^5 f_{j}}{\partial w_{i_1}\cdots\partial w_{i_5}}\bigg(\frac{\mathbf{w}}{\sqrt{n}}\bigg)\bigg|\leq \frac{A_5}{\sqrt{n}}\bigg(1+\sum_{i=1}^{d}|w_i|^{r_5}\bigg).
\end{align*}
Then we have that
\begin{align}
	\bigg|\frac{\partial^5 g_{4,j}}{\partial w_{i_1}\cdots\partial w_{i_5}}(\mathbf{w})\bigg|^{6/5}&\leq\bigg|\frac{1}{\sqrt{n}}\frac{\partial^5 f_{j}}{\partial w_{i_1}\cdots\partial w_{i_5}}\bigg(\frac{\mathbf{w}}{\sqrt{n}}\bigg)\bigg|^{6/5}\leq\frac{A_5^{6/5}}{n^{3/5}}\bigg(1+\sum_{i=1}^{d}|w_i|^{r_5}\bigg)^{6/5}\nonumber\\
	&\leq\frac{2^{1/5}A_5^{6/5}}{n^{3/5}}d^{1/5}\bigg(1+\sum_{i=1}^{d}|w_i|^{6r_5/5}\bigg)\nonumber\\
	&\leq2{C}_{2,4}d^{17}\bigg(d+\sum_{i=1}^{d}|w_i|^{u_{2,4}}\bigg).\label{4dom5bound}
\end{align}
We also have the following bound for the sixth order partial derivatives:
\begin{align}
	\bigg|\frac{\partial^6 g_{4,j}}{\partial w_{i_1}\cdots\partial w_{i_6}}(\mathbf{w})\bigg|&\leq\bigg|\frac{1}{n}\frac{\partial^6 f_{j}}{\partial w_{i_1}\cdots\partial w_{i_6}}\bigg(\frac{\mathbf{w}}{\sqrt{n}}\bigg)\bigg|\leq\frac{A_6}{n}\bigg(1+\sum_{i=1}^{d}|w_i|^{r_6}\bigg)\nonumber\\
	&\leq2{C}_{2,4}d^{17}\bigg(d+\sum_{i=1}^{d}|w_i|^{u_{2,4}}\bigg).\label{4dom6bound}
\end{align}
From inequalities (\ref{2.2dom1bound}), (\ref{2.2dom2bound}), (\ref{2.2dom3bound}), (\ref{2.2dom4bound}), (\ref{4dom5bound}) and (\ref{4dom6bound}), we conclude that $\mathbf{g}_{4,n}\in C_P^{6,m}(\mathbb{R}^d).$

Finally, we consider the case $t=2$. Inequality (\ref{2.2dom1bound}) still holds for $t=2$. It was shown in the proof of Theorem \ref{thm2.1} that (\ref{2.12bound}) holds for $t=2$, and this then implies that (\ref{2.2dom2bound}) also holds for $t=2$. Differentiating $\mathbf{g}_{2,n}(\mathbf{w})=n(\mathbf{f}(n^{-1/2}\mathbf{w})-\mathbf{f}(\mathbf{0}))$ and using that $\mathbf{f}\in \mathcal{C}^{2,4,m}_P(\mathbb{R}^d)$ leads to the following bounds for the fourth, fifth and sixth order partial derivatives: 
\begin{align}
	\bigg|\frac{\partial^4g_{2,j}}{\partial w_{i_1}\cdots\partial w_{i_4}}(\mathbf{w})\bigg|&=\bigg|\frac{1}{n}\frac{\partial^4f_{j}}{\partial w_{i_1}\cdots\partial w_{i_4}}\bigg(\frac{\mathbf{w}}{\sqrt{n}}\bigg)\bigg|\leq\frac{A_4}{n}\bigg(1+\sum_{i=1}^{d}|w_i|^{r_4}\bigg),\label{2four}\\
	\bigg|\frac{\partial^5g_{2,j}}{\partial w_{i_1}\cdots\partial w_{i_5}}(\mathbf{w})\bigg|&=\bigg|\frac{1}{n^{3/2}}\frac{\partial^5f_{j}}{\partial w_{i_1}\cdots\partial w_{i_5}}\bigg(\frac{\mathbf{w}}{\sqrt{n}}\bigg)\bigg|\leq\frac{A_5}{n^{3/2}}\bigg(1+\sum_{i=1}^{d}|w_i|^{r_5}\bigg)\label{2five},\\
	\bigg|\frac{\partial^6g_{2,j}}{\partial w_{i_1}\cdots\partial w_{i_6}}(\mathbf{w})\bigg|&=\bigg|\frac{1}{n^2}\frac{\partial^6f_{j}}{\partial w_{i_1}\cdots\partial w_{i_6}}\bigg(\frac{\mathbf{w}}{\sqrt{n}}\bigg)\bigg|\leq\frac{A_6}{n^2}\bigg(1+\sum_{i=1}^{d}|w_i|^{r_6}\bigg).\label{2six}
\end{align}
From inequalities (\ref{2.2dom1bound}) and (\ref{2.2dom2bound}) and applying similar considerations to those used throughout this proof to the inequalities (\ref{2third}), (\ref{2four}), (\ref{2five}) and (\ref{2six}), it readily follows that $\mathbf{g}_{2,n}\in C_P^{6,m}(\mathbb{R}^d)$. We have thus finally shown that $\mathbf{g}_{t,n}\in C_P^{6,m}(\mathbb{R}^d)$ for $t\geq2$.\qed

\vspace{2mm}

\noindent{\bf{Verification for the proof of Theorem \ref{thm2.3}.}}	Again, we assume that $d/n^{r_t/2}\leq1$ to simplify notation with the general case easily following. We first consider the case $t\geq4$. Applying the inequality $\big(|a_1|+\cdots+|a_q|\big)^s\leq q^{s-1}\big(|a_1|^s+\cdots+|a_q|^s\big),$ for $q,s\geq1$, to the bound (\ref{2.11bound}) yields 
\begin{align}
\bigg|\frac{\partial g_{t,j}}{\partial w_{i_1}}(\mathbf{w})\bigg|^4&\leq\frac{A_t^4}{((t-1)!)^4}d^{4(t-2)}\bigg(\sum_{i=1}^{d}\big(|w_i|^{t-1}+|w_i|^{r_t+t-1}\big)\bigg)^4\nonumber\\
&\leq\frac{8A_t^4}{((t-1)!)^4}d^{4t-5}\sum_{i=1}^{d}\big(|w_i|^{4(t-1)}+|w_i|^{4(r_t+t-1)}\big)\nonumber\\
&\leq2{C}_{3,t}d^{4t-5}\bigg(d+\sum_{i=1}^{d}|w_i|^{u_{3,t}}\bigg).\label{2.3dom1bound}
\end{align}
By applying similar considerations to the bounds (\ref{2.12bound}), (\ref{2.13bound}) and (\ref{2.24bound}), we obtain
\begin{align}
	\bigg|\frac{\partial^2 g_{t,j}}{\partial w_{i_1}\partial w_{i_2}}(\mathbf{w})\bigg|^2&\leq\frac{A_t^2}{((t-2)!)^2}d^{2(t-3)}\bigg(\sum_{i=1}^{d}\big(|w_i|^{t-2}+d|w_i|^{r_t+t-2}\big)\bigg)^2\nonumber\\
	&\leq\frac{2A_t^2}{((t-2)!)^2}d^{2t-5}\sum_{i=1}^{d}\big(|w_i|^{2(t-2)}+d^2|w_i|^{2(r_t+t-2)}\big)\nonumber\\
	&\leq2{C}_{3,t}d^{4t-5}\bigg(d+\sum_{i=1}^{d}|w_i|^{u_{3,t}}\bigg),\label{2.3dom2bound}\\
	\bigg|\frac{\partial^3 g_{t,j}}{\partial w_{i_1}\partial w_{i_2}\partial w_{i_3}}(\mathbf{w})\bigg|^{4/3}&\leq\frac{A_t^{4/3}}{((t-3)!)^{4/3}}d^{4(t-4)/3}\bigg(\sum_{i=1}^{d}\big(|w_i|^{t-3}+d|w_i|^{r_t+t-3}\big)\bigg)^{4/3}\nonumber\\
	&\leq\frac{2^{1/3}A_t^{4/3}}{((t-3)!)^{4/3}}d^{4t/3-5}\sum_{i=1}^{d}\big(|w_i|^{4(t-3)/3}+d^{4/3}|w_i|^{4(r_t+t-3)/3}\big)\nonumber\\
	&\leq2{C}_{3,t}d^{4t-5}\bigg(d+\sum_{i=1}^{d}|w_i|^{u_{3,t}}\bigg)\label{2.3dom3bound},\\
	\bigg|\frac{\partial^4 g_{4,j}}{\partial w_{i_1}\cdots\partial w_{i_4}}(\mathbf{w})\bigg|&\leq\frac{A_t}{(t-4)!}d^{t-5}\sum_{i=1}^{d}\big(|w_i|^{t-4}+d|w_i|^{r_t+t-4}\big)\nonumber\\
	&\leq2{C}_{3,t}d^{4t-5}\bigg(d+\sum_{i=1}^{d}|w_i|^{u_{3,t}}\bigg).\label{2.3dom4bound}
	\end{align}
From (\ref{2.3dom1bound}), (\ref{2.3dom2bound}), (\ref{2.3dom3bound}) and (\ref{2.3dom4bound}), we deduce that $\mathbf{g}_{t,n}\in C_P^{4,m}(\mathbb{R}^d)$ for $t\geq4$. 

We now consider the case $t=2$. Inequality (\ref{2.3dom1bound}) holds for $t=2$, but inequalities (\ref{2.3dom2bound}), (\ref{2.3dom3bound}) and (\ref{2.3dom4bound}) were derived under the assumption that $t\geq4$. It was shown in the proof of Theorem \ref{thm2.1} that (\ref{2.12bound}) is valid for $t=2$, which in turn implies that the bound (\ref{2.3dom2bound}) is valid for $t=2$. Differentiating $\mathbf{g}_{2,n}(\mathbf{w})=n\big(\mathbf{f}(n^{-1/2}\mathbf{w})-\mathbf{f}(\mathbf{0})\big),$ we have 
\begin{align*}
	\bigg|\frac{\partial^3 g_{t,j}}{\partial w_{i_1}\partial w_{i_2}\partial w_{i_3}}(\mathbf{w})\bigg|=\bigg|\frac{1}{\sqrt{n}}\frac{\partial^3 f_{j}}{\partial w_{i_1}\partial w_{i_2}\partial w_{i_3}}\bigg(\frac{\mathbf{w}}{\sqrt{n}}\bigg)\bigg|\leq\frac{A_3}{\sqrt{n}}\bigg(1+\sum_{i=1}^{d}|w_i|^{r_3}\bigg),
\end{align*}
which then gives
\begin{align}
	\bigg|\frac{\partial^3 g_{t,j}}{\partial w_{i_1}\partial w_{i_2}\partial w_{i_3}}(\mathbf{w})\bigg|^{4/3}&\leq\frac{A_3^{4/3}}{n^{2/3}}\bigg(1+\sum_{i=1}^{d}|w_i|^{r_3}\bigg)^{4/3}\leq\frac{2^{1/3}A_3^{4/3}}{n^{2/3}}\bigg(1+d^{1/3}\sum_{i=1}^{d}|w_i|^{4r_3/3}\bigg)\nonumber\\
	&\leq 2{C}_{3,2}d^3\bigg(d+\sum_{i=1}^{d}|w_i|^{u_{3,2}}\bigg).\label{2.32dom3bound}\end{align}
We also have the following bound for the fourth order partial derivatives:
\begin{align}
		\bigg|\frac{\partial^4 g_{4,j}}{\partial w_{i_1}\cdots\partial w_{i_4}}(\mathbf{w})\bigg|&=\bigg|\frac{1}{n}\frac{\partial^4 f_{j}}{\partial w_{i_1}\cdots\partial w_{i_4}}(\mathbf{w})\bigg|\leq\frac{A_4}{n}\bigg(1+\sum_{i=1}^{d}|w_i|^{r_4}\bigg)\nonumber\\
		&\leq2{C}_{3,2}d^3\bigg(d+\sum_{i=1}^{d}|w_i|^{u_{3,2}}\bigg).\label{2.32dom4bound}
\end{align}
From inequalities (\ref{2.3dom1bound}), (\ref{2.3dom2bound}), (\ref{2.32dom3bound}) and (\ref{2.32dom4bound}), we have that $\mathbf{g}_{2,n}\in C_p^{4,m}(\mathbb{R}^d)$. We have thus finally shown that $\mathbf{g}_{t,n}\in C_p^{4,m}(\mathbb{R}^d)$ for $t\geq2$. \qed

\vspace{2mm}

\noindent{\bf{Verification for the proof of Theorem \ref{thm2.1univ}.}}	  By first Taylor expanding $f(n^{-1/2}w)$ about 0 and using that $f'(0)=\cdots=f^{(t-2)}(0)=0$ (in the case $t\geq2$), and then differentiating, we obtain 
\begin{equation}
	g'_{t,n}(w)=\frac{w^{t-1}}{(t-1)!}f^{(t)}\bigg(\frac{\theta w}{\sqrt{n}}\bigg)\nonumber,
\end{equation}
where $\theta\in(0,1)$.
Since $f\in\mathcal{C}_P^{t,1,1}(\mathbb{R}),$ we obtain the bound 
\begin{equation}
	|g'_{t,n}(w)|\leq\frac{A_t}{(t-1)!}\bigg(|w|^{t-1}+\frac{1}{n^{r_t/2}}|w|^{r_t+t-1}\bigg)\leq \frac{2A_t}{(t-1)!}(1+|w|^{r_t+t-1}).\label{2.4g'bound}
\end{equation}
Inequality (\ref{2.4g'bound}) was derived under the assumption that $t\geq2$, but we can show that it is also valid for $t=1$. Differentiating $g_{1,n}(w)=\sqrt{n}(f(n^{-1/2}w)-f(0))$ gives that 
\begin{equation}
	|g'_{1,n}(w)|=|f'(n^{-1/2}w)|\leq A_1\bigg(1+\frac{1}{n^{r_1/2}}|w|^{r_1}\bigg)\leq 2A_1(1+|w|^{r_1})\nonumber,
\end{equation}
and so (\ref{2.4g'bound}) is valid for $t\geq1$. Thus, $g_{t,n}\in C_P^{1,1}(\mathbb{R})$, as required.
 \qed

\vspace{2mm}

\noindent{\bf{Verification for the proof of Theorem \ref{thm2.2univ}.}}	Applying the inequality $(a+b)^2\leq 2(a^2+b^2)$ to (\ref{2.4g'bound}) and our usual considerations yields
\begin{align}
	|g'_{t,n}(w)|^2
\leq\frac{2A_t^2}{((t-1)!)^2}(|w|^{2(t-1)}+|w|^{2(r_t+t-1)})\leq\frac{4A_t^2}{((t-1)!)^2}(1+|w|^{2(r_t+t-1)})\label{2.5g'bound}.
\end{align}
It now remains to bound the second derivative of $g_{t,n}$. By first Taylor expanding $f(n^{-1/2}w)$ about 0 and using that $f'(0)=\cdots=f^{(t-3)}(0)=0$, and then differentiating, we obtain that, for $t\geq3$,
\begin{equation}
	g''_{t,n}(w)=\frac{w^{t-2}}{(t-2)!}f^{(t)}\bigg(\frac{\theta w}{\sqrt{n}}\bigg).\nonumber
\end{equation}
Since $f\in\mathcal{C}_P^{t,2,1}(\mathbb{R}),$ we obtain the bound
\begin{equation}
	|g''_{t,n}(w)|\leq\frac{A_t}{(t-2)!}\bigg(|w|^{t-2}+\frac{1}{n^{r_t/2}}|w|^{r_t+t-2}\bigg)\label{2.5g''star}.
\end{equation}
By our usual arguments, it can be seen that inequality (\ref{2.5g''star}) is also valid for $t=2$. Applying the basic inequalities $|w|^{t-2}\leq 1+|w|^{2(r_t+t-1)}$ and $|w|^{r_t+t-2}\leq 1+|w|^{2(r_t+t-1)}$ to (\ref{2.5g''star}) yields
\begin{equation}
	|g''_{t,n}(w)|\leq \frac{2A_t}{(t-2)!}(1+|w|^{2(r_t+t-1)})\label{2.5g''bound}.
\end{equation}
From inequalities (\ref{2.5g'bound}) and (\ref{2.5g''bound}), we deduce that $g_{t,n}\in C_P^{2,1}(\mathbb{R})$ for $t\geq2$. \qed


\begin{thebibliography}{99}
\bibitem{pearson} Pearson, K.  On the criterion that a given system of deviations is such that it can be reasonably supposed to have arisen from random sampling.  \emph{Phil. Mag.} $\mathbf{50}$ (1900),  157--175.
\end{thebibliography}
	\end{document}